\documentclass[a4paper]{article}
\usepackage[T2A]{fontenc}
\usepackage[english]{babel}

\usepackage{amsmath, amsthm}
\usepackage{amssymb}
\usepackage{amsfonts}
\usepackage{srcltx, dsfont}

\theoremstyle{plain}
\newtheorem{theorem}{Theorem}[section]
\newtheorem{lemma}{Lemma}[section]

\numberwithin{equation}{section}

\def\Tht{\Theta}
\def\tht{\theta}
\def\Om{\Omega}
\def\om{\omega}
\def\e{\varepsilon}
\def\g{\gamma}

\def\l{\lambda}
\def\p{\partial}
\def\D{\Delta}
\def\a{\alpha}

\def\d{\delta}
\def\L{\Lambda}
\def\z{\zeta}

\def\vt{\vartheta}
\def\vk{\varkappa}

\def\H{W_2}
\def\Ho{\mathring{W}_2}

\def\hf{\mathfrak{h}}

\def\di{\,d}

\def\Op{\mathcal{H}}

\def\x{\mathrm{x}}
\def\V{\mathrm{V}}

\def\lf{\lfloor}
\def\rf{\rfloor}

\def\cE{\mathcal{E}}

\def\cL{\mathcal{L}}

\DeclareMathOperator{\Tr}{Tr}
\DeclareMathOperator{\dvr}{div}

\DeclareMathOperator{\Dom}{\mathfrak{D}}

\DeclareMathOperator{\dist}{dist}
\DeclareMathOperator{\supp}{supp}

\allowdisplaybreaks

\begin{document}

\title{Asymptotic analysis of exit time for dynamical systems with a single well potential}

\author{D. Borisov$^{1,2,3}$, O. Sultanov$^1$}

\date{\empty}

\maketitle

{\small
    \begin{quote}
1) Institute of Mathematics, Ufa Federal Research Center, Russian Academy of Sciences, \newline Chernyshevsky str. 112, Ufa, Russia, 450008
\\
2) Bashkir State  University, Zaki Validi str. 32, Ufa, Russia, 450000
\\
3) University of Hradec Kr\'alov\'e
62, Rokitansk\'eho, Hradec Kr\'alov\'e 50003, Czech Republic
\\
Emails: borisovdi@yandex.ru, oasultanov@gmail.com
\end{quote}

{\small
    \begin{quote}
    \noindent{\bf Abstract.} We study the exit time   from a bounded multi-dimensional domain $\Om$
 of the
stochastic process $\mathbf{Y}_\e=\mathbf{Y}_\e(t,a)$, $t\geqslant 0$, $a\in \mathcal{A}$, governed by the overdamped Langevin dynamics
\begin{equation*}
 d\mathbf{Y}_\e  =-\nabla V(\mathbf{Y}_\e) dt +\sqrt{2}\e\, d\mathbf{W}, \qquad \mathbf{Y}_\e(0,a)\equiv x\in\Om
\end{equation*}
where $\e$ is a small positive parameter, $\mathcal{A}$ is a sample space, $\mathbf{W}$ is a $n$-dimensional Wiener process. The exit time corresponds to the first hitting of $\p\Om$ by the trajectories of the above dynamical system and the expectation value of this exit time solves the boundary value problem
\begin{equation*}
(-\e^2\D +\nabla V\cdot \nabla)u_\e=1\quad\text{in}\quad\Om,\qquad u_\e=0\quad\text{on}\quad\p\Om.
\end{equation*}
We assume that the function $V$ is smooth enough  and has the only minimum at the origin (contained in $\Om$); the minimum  can be degenerate. At other points of $\Om$, the gradient of $V$ is non-zero and the normal derivative of $V$ at the boundary $\p\Om$ does not vanish as well. Our main result is a complete asymptotic expansion for $u_\e$ as well as for the lowest eigenvalue of the considered problem and for the associated eigenfunction. The asymptotics for $u_\e$ involves a term exponentially large $\e$;  we find this term in a closed form. Apart of this term, we also construct a power in $\e$ asymptotic expansion such that this expansion and a mentioned exponentially large term approximate $u_\e$ up to arbitrarily  power of $\e$. We also discuss some probabilistic aspects of our results.

    \medskip


    \medskip

 	
    \end{quote}
}

\section{Introduction}

Equations with a small parameter at higher derivatives is one of the classical directions in the  modern mathematical physics. Such equations arise in various applications and one of them is the exit time problem. Originally, the problem is introduced as a system of It\^o stochastic differential equations~\cite[Sect.~4.1]{Schuss}:
\begin{equation}\label{0.1}
 d\mathbf{Y}_\e(t,a) =\mathbf{F}(\mathbf{Y}_\e(t,a)) dt +\sqrt{2}\e\, d\mathbf{W}, \quad \mathbf{Y}_\e(0,a)\equiv   x,\quad a\in \mathcal{A},
\end{equation}
where $\mathbf{Y}_\e$ is a $n$-dimensional vector, $\mathbf{F}:\,\mathds{R}^n\to\mathds{R}^n$ is a given function. The origin is assumed to be an asymptotically stable equilibrium of equation (\ref{0.1}) as $\e=0$. The second term in the right hand side in (\ref{0.1}) serves as a random perturbation. Namely, $\mathbf{W}=(\mathrm{W}_1(t,a),\dots, \mathrm{W}_n(t,a))$ is a $n$-dimensional Wiener process  on a probability
space $(\mathcal{A}, \mathcal{F},\mathbb{P})$, where  $\mathcal{A}=\{a\}$ is the sample space, $\mathcal{F}$ is a $\sigma$-algebra,  $\mathbb{P}$ is a probability measure. The symbol $\e$ stands for a small positive parameter characterizing the perturbation strength.

A solution to problem (\ref{0.1}) is a random process $\mathbf{Y}_\e(t,a)$, $t\geqslant 0$, $a\in\mathcal{A}$. It is known~\cite[Ch.~4]{WFbook} that the trajectories $\mathbf{Y}(t,a)$ leave a vicinity of the origin almost surely. This process is characterized by the mean exit time introduced as
\begin{equation*}
    u_\e(x):=\mathbb{E}\big(T_\e(a):\,\mathbf{Y}_\e(0,a)=x\big),
\end{equation*}
where
\begin{equation}\label{0.7}
T_\e(a)=\inf \{t\geqslant 0:\, \mathbf{Y}_\e(t,a)\not\in \Om\}
\end{equation}
is the first time, when the trajectory $\mathbf{Y}(t,a)$ hits the boundary $\p\Om$. The function $u_\e$ solves the boundary value problem~\cite[Sect 4.4]{Schuss}
\begin{equation}
\label{0.2}
-\e^2  \D u_\e - \sum\limits_{i=1}^n F_i(x)\frac{\p u_\e}{\p x_i}=1\quad\text{in}\quad\Om, \qquad u_\e=0 \quad\text{on}\quad \p\Om.
\end{equation}
This problem and similar ones for more general equations were studied in a series works.

In \cite{Levinson1950}, there was considered the Dirichlet problem for the equation
\begin{equation}\label{0.3}
\e^2 \D u+A u_x+ B u_y +C=0
\end{equation}
in a two-dimensional domain. Apart of some smoothness, the main assumptions made for the coefficients of equation (\ref{0.3}) was either $A^2+B^2>0$ or $C<0$. The main result was an asymptotic expansion for the solution but the suggested way of calculating the terms in this asymptotics was quite complicated.

A multi-dimensional case was treated in \cite{MS77}. Here the authors considered the Dirichlet problem for the equation
\begin{equation}\label{0.4}
\e \sum\limits_{i,j=1}^{n} a_{ij} \frac{\p^2 u}{\p x_i\p x_j} + \sum\limits_{i=1}^{n} b_i(x)\frac{\p u}{\p x_i}=f
\end{equation}
in a bounded multi-dimensional domain. There was proposed a scheme  for constructing a formal asymptotic expansion for the solution. No justification was made, that is, no rigorous estimates for the error terms were obtained. A feature of the found formal asymptotic expansion is that it involved a constant in the leading term and this constant was not defined during the formal construction. Another, again formal way based on special integration by parts was proposed to determine this constant. The final formula for the constant stated that this constant is exponentially large as $\e$ tends to zero.
The boundary value problem for equation   (\ref{0.4}) was also considered in \cite{Kamin}. The principal part in the equation was the Laplacian. An important difference with work \cite{MS77} was that in \cite{Kamin} the equation was homogeneous, while the boundary condition involved an arbitrary right hand side. The main result stated that the solution of the considered problem converged to some constant uniformly on compact subsets as $\e\to0$. Since the equation was homogeneous, the constant turned out to be independent of $\e$ and this stresses an important difference between having a right hand in the equation  or in the boundary condition. A similar result but for the Neumann condition on the boundary was obtained in \cite{Pe}. The results of such kind, the convergence of the solution to a constant for homogeneous equation and inhomogeneous boundary conditions, were also obtained in \cite{Day}, \cite{Ishi}, \cite{Kamin2}, \cite{Kamin3}, \cite{Kamin4}. In paper \cite{Ward01}, there were  considered more general non-stationary problems associated with the above discussed equations and the constant in the leading term in the asymptotics was found by analysing the long time behavior of their solutions. In \cite{FK}, a result of such kind was obtained for a quasi-linear equation.

Problem (\ref{0.2}) was intensively studied in the case $F=-\nabla V$, where $V$ is some function having one  or several extrema in $\Om$. Such case corresponds to the overdamped Langevin dynamics.
The case of one extrema was studied in work \cite{Nec1} for an arbitrary multi-dimensional domain. The main result provided the leading term of the asymptotics of the solution on each compact subset of $\Om$ in the form
\begin{equation}\label{0.5}
u_\e(x)=U_\e(x)(1+O(\e)),
\end{equation}
where $U_\e$ was some explicitly function exponentially large in $\e$. It was said clearly in the work that the employed approach did not allow the author to find a complete asymptotic expansion.
 Close results were earlier obtained in \cite{Sug1}, \cite{Sug2}. In these works, the potential $V$ could have several extrema and the problem was considered on a Riemannian manifold. The main result stated that the solution satisfied the identity
\begin{equation*}
\lim\limits_{\e\to+0} \e^{-\mu} e^{-C_0\e^{-2}} u_\e=C_1
\end{equation*}
with some constants $\mu$, $C_0$ and $C_1$ expressed in terms of certain characteristics of the potential and the domain.

The aforementioned case of a potential $V$ with several extrema attracted a lot of attention  and it corresponds to the so-called metastability phenomenon. The matter is that in the case of several local minima of $V$, the Brownian particle can be first attracted to one of them and then to the other and this makes its dynamics more interesting. Apart of \cite{Sug1}, \cite{Sug2}, a lot of papers was devoted to studying this phenomenon and here we mention just some of them. In \cite{JEMS1}, problem (\ref{0.2}) was considered in an arbitrary multi-dimensional domain. It was assumed that the Hessian of $V$ was non-degenerate at its local minima. The main obtained result was the leading term in the asymptotics for the exit time in form similar to (\ref{0.5}); the error term was $O(\e|\ln\e|)$. In works \cite{Nuh1} there was studied   a   quasi-stationary distribution  in the case when $V$ had a double-well structure. The density of this measure was a first eigenfunction  of the considered operator multiplied by $e^{-\frac{V}{\e^2}}$. The main result was the leading term in the asymptotics for this density in the vicinity of the minima of $V$ and rigorous estimates for the error terms. The case of one global minimum of $V$ and several local minimal was addressed in \cite{Nuh3} and the main results was again the leading term of the asymptotics for the aforementioned measure. It was also stated that this measure has a complete asymptotic expansion with some coefficients $b_{k,i}$; the authors said that  ``the explicit computations of the sequence $b_{k,i}$
 is not possible in practice''.
The same and similar measures were studied also in \cite{Nuh2} and the main result was the leading terms in the asymptotics for the expectation values with respect to the mentioned measures. There were also found the conditions, under which a complete asymptotic expansion could be constructed.

We should also  mention that the exit time was also  studied by probabilistic methods, see, for instance, a classical work \cite{VL} and also \cite{Ki}. In \cite{VL}, an exponential lower bound for the   exit time was obtained. In \cite{Ki}, the exit time was characterized by certain probabilistic asymptotic relations.

In a very recent work \cite{LKOSNT17}, there was considered a model situation for  equation (\ref{0.2}) with $F(x)=-x$, when the domain was the circle.  An explicit solution for the considered problem was constructed. On the base of this formula, the asymptotic expansion for the solution was found.

Apart of the asymptotics for the solution, the behavior of the lowest eigenvalue attracted much attention. The reason is that the reciprocal to this eigenvalue, as well as the structure of the associated eigenfunction played important role in determining various characteristics of random process $\mathbf{Y}$ in (\ref{0.1}).
Such asymptotics was formally constructed in \cite{MS77} for the operator in (\ref{0.4}). The formal construction suggested that this eigenvalue is exponentially close to zero. It should be also said that the asymptotics for the eigenvalues of the operators with a small parameter at higher derivatives were widely studied by many authors. Here we cite only some classical works \cite{HMR}, \cite{HBook}, \cite{QC}, where first rigorous results were proved.
The problem on  complete  asymptotic expansions for the lowest eigenvalues and the associated eigenfunctions of problem (\ref{0.2}) with $F=-\nabla V$ was finally solved in \cite{HN}; the operator was considered on a Riemannian manifold with a boundary and the function $V$ was assumed to be a Morse one.

In the present work we consider problem (\ref{0.2}) in an arbitrary bounded multi-dimensional domain with an infinitely smooth boundary. We assume that $F=-\nabla V$, where $V$ is an arbitrary sufficiently smooth function with the only minimum at the origin and non-zero gradient in  $\Om\setminus\{0\}$. The minimum at the origin can be degenerate with an arbitrary fixed power rate, so, the function $V$ is not necessary to be a Morse one. The detailed assumption on $V$ are formulated in (\ref{2.1a}), (\ref{2.6a}). Our main result is the complete asymptotic expansion for the solution of the considered problem. The structure of the asymptotics is as follows:
\begin{equation}\label{0.6}
u_\e=K_\e\Psi_\e + \textit{boundary layer} + \textit{error term},
\end{equation}
where $\Psi_\e$ is the eigenfunction associated with the lowest eigenvalue of the considered problem, $K_\e$ is a constant, the boundary layer is power in $\e$ and the error is also power. The constant $K_\e$ is represented a sum of two terms, $K_\e^{(exp)}$ and $K_\e^{(pow)}$. The former is exponentially large and we find its complete asymptotic expansion as $e^{\frac{ \min_{\p\Om} V}{\e^2}}$  times a power series in $\e$, while $K_\e^{(pow)}$ has a complete power in $\e$ asymptotic expansion. The error term in (\ref{0.6}) is estimated in various norms, both on the entire domain and compact subdomains.  We also construct a complete expansion for the eigenfunction $\Psi_\e$. The latter is justified in the same norms as (\ref{0.6}). An important feature of our result is that we provide a straightforward and rather simple algorithm of calculating all coefficients in the aforementioned asymptotic expansions. As an example, we find explicitly two terms in all above asymptotics.

The paper is organized as follows. In the next section we formulate the problem and present the main results. In the third section we discuss the results from the probabilistic point of view. The next two section are devoted to finding the asymptotics for the lowest eigenvalue and the associated eigenfunction of the considered problem. In the last section we construct the asymptotics of the solution to problem (\ref{0.2}) and   prove the main result.

\section{Problem and main results}

Let $x=(x_1,\ldots,x_n)$ be Cartesian coordinates in $\mathds{R}^n$, $n\geqslant 2$,
$\Om\subset\mathds{R}^n$ be a bounded domain containing the origin and having an infinitely differentiable boundary, $\x=\x_1(s)=(\x_1(s),\ldots,\x_n(s))$ be the vector equation of $\p\Om$, where $s$ are local coordinates on $\p\Om$ associated with some atlas. Thanks to  the assumed smoothness of $\p\Om$, the functions $\x_i(s)$ are infinitely differentiable.
By $\nu=\nu(s)$ we denote the inward normal to $\p\Om$.

In the vicinity of $\p\Om$ we introduce local coordinates $(s,\tau)$, where $\tau$ is the distance to a point measured along the inward normal $\nu$, that is, $x=\x(s)+\tau\nu(s)$. Since the boundary $\p\Om$ is infinitely differentiable,  the introduced local coordinates are well-defined up to $|\tau|\leqslant \tau_0$ for some sufficiently small fixed $\tau_0$.

By $V=V(x)$ we denote a real function defined on $\Om$ satisfying two main assumptions. The first assumption is on the smoothness:
\begin{equation}\label{2.1a}
 V\in C^2(\overline{\Om}),\quad V\in C^\infty(\overline{B_{\rho_1}}),\qquad V\in C^\infty(\{x:\, 0\leqslant \tau\leqslant \tau_0\}),\qquad d^k V(0)\ne0,
\end{equation}
where  $\rho_1>0$,
$k\in\mathds{N}$, $k\geqslant 2$, are some fixed  constants and $B_\rho$ stands for an open ball or a radius $\rho$ centered at the origin.  The second assumption is as follows:
\begin{equation}\label{2.6a}
\begin{aligned}
&V(0)=0,\qquad \nabla V(0)=0,\qquad  V(x)>0
\quad\text{as}\quad x\in\overline{\Om}\setminus\{0\},
\\
&
 |\nabla V(x)|\geqslant c_1>0
\quad\text{as}\quad x\in\overline{\Om}\setminus B_{\rho_1},
\qquad \frac{\p V}{\p\tau}\leqslant -c_2<0 \quad\text{as}\quad 0\leqslant \tau\leqslant \tau_0,
\end{aligned}
\end{equation}
with some positive constants   $c_1$, $c_2$
independent of $x$.

The main object of our study is an unbounded operator $\Op_\e$ in the space $L_2(\Om)$ with the differential expression
\begin{equation*}
-\e^2\D + \nabla V(x)\cdot \nabla
\end{equation*}
on the domain $\Dom(\Op_\e):=\Ho^2(\Om)$. Hereinafter $\e$ is  a small positive parameter, and   $\Ho^j(\Om)$, $j\geqslant 1$, is the Sobolev space of functions in $\H^j(\Om)$ with the zero trace on $\p\Om$.
The operator $\Op_\e$ is closed. It has a compact resolvent and its spectrum consists of countably many discrete eigenvalues accumulating at infinity only.  Our main aim is to find the asymptotic expansion for the solution to the equation
\begin{equation}\label{3.2}
\Op_\e u_\e=\mathds{1},
\end{equation}
where the symbol $\mathds{1}$ stands for the constant function $\mathds{1}(x)\equiv 1$ on $\Om$.

Before formulating the main results,  we introduce additional assumptions and auxiliary notations.

By
\begin{equation*}
g=
\begin{pmatrix}
g_{11} & \ldots & g_{1\,n-1}
\\
\vdots & \vdots &
\\
g_{n-1\,1} & \vdots & g_{n-1\,n-1}
\end{pmatrix},\qquad b=
\begin{pmatrix}
b_{11} & \ldots & b_{1\,n-1}
\\
\vdots & \vdots &
\\
b_{n-1\,1} & \vdots & b_{n-1\,n-1}
\end{pmatrix}, \qquad g_{ij}=g_{ij}(s),\qquad b_{ij}=b_{ij}(s),
\end{equation*}
we denote the metric tensor on $\p\Om$ and the second fundamental form on the inward side of $\p\Om$.  We define
\begin{equation}\label{2.3}
\begin{aligned}
&\tht_0(s):=V(\x(s)), && \tht_{min}:=\min\limits_{\p\Om} \tht_0(s), && \tht_1(s):=\frac{\p V}{\p\tau}\bigg|_{\tau=0}\leqslant -c_2<0,
\\
&\tht_2(s):=\frac{1}{2} \frac{\p^2 V}{\p\tau^2}\bigg|_{\tau=0},\qquad && \Tht_0:=\ln\sqrt{\det g},\qquad && \Tht_1:=-\sqrt{\det g}\Tr b.
\end{aligned}
\end{equation}

Let $\chi_0=\chi_0(t)$ be an infinitely differentiable function vanishing as $t>2$ and equalling to one as $t<1$. In $\Om$ we define a function
\begin{equation}\label{3.13a}
\chi(x):=\left\{
\begin{aligned}
\chi_0& (3\tau\d^{-1}) &&\text{as}\quad 0\leqslant \tau\leqslant \d,
\\
&\ 0&&\text{otherwise},
\end{aligned}
\right.\qquad \d<\min\left\{\tau_0,\frac{\tht_{min}}{4c_2}\right\}.
\end{equation}
For each subdomain $\om\subset\Om$ we denote $\V_\om:=\sup\limits_{\overline{\om}} V(x)$. We  also introduce the shorthand notation
\begin{equation*}
\|\p^2_{xx} u\|_{L_2(\om)}=\sum\limits_{i,j=1}^{2}\left\|\frac{\p^2  u}{\p x_i \p x_j}\right\|_{L_2(\om)}\quad \text{for}\quad u\in\H^2(\Om).
\end{equation*}
By $\lf\cdot\rf$ we denote an integer part of a number.

Now we are in position to formulate our   results. The first of them is devoted to the spectrum of the operator $\Op^\e$.

\begin{theorem}\label{thEF}
All eigenvalues of the operator $\Op_\e$ are  real and the lowest eigenvalue $\l_\e$ is simple.
The eigenfunction of the operator $\Op_\e$ associated with the lowest eigenvalue $\l_\e$ can be chosen so that   it satisfies the asymptotic formula
\begin{equation}
\Psi_\e(x)=1-\chi(x)e^{\frac{\tht_1(s)\tau}{\e^2}} \sum\limits_{j=0}^{N}\e^{2j} \Phi_j(\tau\e^{-2},s)  + \Xi_{\e,N}(x),\qquad N\in\mathds{Z}_+.
\label{2.10}
\end{equation}
The symbols $\Phi_j=\Phi_j(\z,s)$ denote some polynomials in $\z$ of degree at most $2j$ with infinitely differentiable in $s$ coefficients such that $\Phi_j(0,s)=0$, $j\geqslant 1$. In particular,
\begin{equation}\label{2.12}
\Phi_0(\z,s)\equiv 1, \qquad
\Phi_1(\z,s)=\frac{\Phi(s)}{2\tht_1(s)}\z^2 - \left(\Tht_1 + \frac{\Phi(s)}{\tht_1^2(s)}\right)\z,
\qquad \Phi:= g^{-1}\nabla_s (\tht_0-\Tht_0)\cdot \nabla_s \tht_1+2\tht_1\tht_2.
\end{equation}
The error term obeys the estimates
\begin{equation}
\begin{aligned}
&\|\Xi_{\e,N}\|_{L_2(\om)}=O\Big(\e^{2N+3} e^{-\frac{\tht_{min}-\V_\om}{2\e^2}} \Big), && \|\nabla \Xi_{\e,N}\|_{L_2(\om)}=O\Big(\e^{2N+1} e^{-\frac{\tht_{min}-\V_\om}{2\e^2}} \Big), \\ &\|\p^2_{xx}  \Xi_{\e,N}\|_{L_2(\om)} = O\Big(\e^{2N-1} e^{-\frac{\tht_{min}-\V_\om}{2\e^2}}\Big),\qquad &&\|  \Xi_{\e,N}\|_{C(\overline{\om})}=O\Big(\e^{2N+2} e^{-\frac{\tht_{min}-\V_\om}{2\e^2}}\Big),
\end{aligned}\label{2.14a}
\end{equation}
for each subdomain $\om\subset\Om$ and
\begin{align}\label{2.15b}
&
\begin{aligned}
&\|\Xi_{\e,N}\|_{L_2(\Om)}=O(\e^{2N+3}), &&\quad \|\nabla\Xi_{\e,N}\|_{L_2(\Om)}=O(\e^{2N+1}), \\ &\|\p^2_{xx}\Xi_{\e,N}\|_{L_2(\Om)} = O(\e^{2N-1}), && \quad \|\Xi_{\e,N}\|_{C(\overline{\Om})}=O(\e^{2N+2}),
\end{aligned}
\\
&\begin{aligned}
&\|e^{-\frac{V}{2\e^2}}\Xi_{\e,N}\|_{L_2(\om)}=O\Big(\e^{2N+3} e^{-\frac{\tht_{min} }{2\e^2}}\Big), && \|\nabla e^{-\frac{V}{2\e^2}}\Xi_{\e,N}\|_{L_2(\om)}=O\Big(\e^{2N+1} e^{-\frac{\tht_{min} }{2\e^2}}\Big), \\ &\|\p^2_{xx}  e^{-\frac{V}{2\e^2}} \Xi_{\e,N}\|_{L_2(\om)} = O\Big(\e^{2N-1}e^{-\frac{\tht_{min} }{2\e^2}}\Big),\qquad &&\|  e^{-\frac{V}{2\e^2}} \Xi_{\e,N}\|_{C(\overline{\om})}=O\Big(\e^{2N+2} e^{-\frac{\tht_{min} }{2\e^2}}\Big).
\end{aligned}\label{2.15c}
\end{align}
\end{theorem}

The next theorem describes the asymptotic behavior of the lowest eigenvalue.

\begin{theorem}\label{thEV}
The lowest eigenvalue $\l_\e$ of the operator $\Op_\e$ satisfies the identity
\begin{equation}\label{5.14a}
\l_\e=\frac{\e^2\int\limits_{\p\Om} e^{-\frac{V}{\e^2}} \frac{\p\Psi_\e}{\p\tau}\di s} {\int\limits_{\Om} e^{-\frac{V}{\e^2}} \Psi_\e\di x}
\end{equation}
and the asymptotic formula
\begin{equation}\label{2.9}
\l_\e=\e^{-\frac{2n}{k}-2}\sum\limits_{j=0}^{N} \e^{\frac{2j}{k}} M_j\left(\mu_0(\e),\ldots, \mu_{\lf\frac{j}{k}\rf}(\e)
\right) +O\Big(\e^{\frac{2(N+1-n)}{k}-2} e^{-\frac{\tht_{min}}{\e^2}}\Big),\qquad N\in\mathds{Z}_+,
\end{equation}
where $M_j$ are some linear combinations with fixed coefficients, while the functions $\mu_j(\e)$ are defined by the formulae
\begin{equation}
\label{5.23}
\mu_0(\e):= -\int\limits_{\p\Om} e^{-\frac{\tht(s)}{\e^2}} \tht_1(s) \di s,\qquad \mu_j(\e):=-\int\limits_{\p\Om} e^{-\frac{\tht(s)}{\e^2}} \frac{\p \Phi_j}{\p\z}(0,s)\di s,\qquad j\geqslant 1,
\end{equation}
and obey the estimates
\begin{equation}
\mu_j(\e)=O\big(e^{-\frac{\tht_{min}}{\e^2}}\big),\qquad j\geqslant 0.\label{5.24}
\end{equation}

 In particular,
\begin{align}
&\label{5.28}
M_0(\mu_0)=\frac{\mu_0}{ \int\limits_{\mathds{R}^n} e^{-V_0}\di x},
\qquad M_1(\mu_0)=-\frac{\mu_0}{ \left( \int\limits_{\mathds{R}^n} e^{-V_0}\di x\right)^2} \int\limits_{\mathds{R}^n} V_1e^{-V_0}\di x,
\\
&\label{5.28a}
V_i(x):=\frac{1}{(k+i)!}\sum\limits_{
\substack{p\in\mathds{Z}_+^n,\\ |p|=k+i}}\frac{\p^p V}{\p x^p}(0) x^\a,\qquad V_0(x)>0,\quad x\ne0.
\end{align}
\end{theorem}

Our  main result states the solvability of equation (\ref{3.2}) and provides the complete asymptotic expansion for the solution.

\begin{theorem}\label{thEq}
Equation (\ref{3.2}) is uniquely solvable for sufficiently small $\e$. Its solution satisfies the asymptotic formula
\begin{align}\label{2.13}
&u_\e(x)=K_\e\Psi_\e(x)+\chi(x) e^{ \frac{\tht_1(s)\tau}{\e^2}} \sum\limits_{j=1}^{N}\e^{2j} U_j(\tau\e^{-2},s) + S_N^\e(x),\qquad N\in\mathds{N},
\end{align}
where the error term obeys the estimates
\begin{equation}\label{2.14}
\begin{aligned}
&\|S_N^\e\|_{L_2(\om)}=O\Big(\e^{2N+3} e^{-\frac{\tht_{min}-\V_\om}{2\e^2}} \Big), && \|\nabla S_N^\e\|_{L_2(\om)}=O\Big(\e^{2N+1} e^{-\frac{\tht_{min}-\V_\om}{2\e^2}} \Big),
\\
&\|\p^2_{xx}  S_N^\e\|_{L_2(\om)} = O\Big(\e^{2N-1} e^{-\frac{\tht_{min}-\V_\om}{2\e^2}}\Big),\qquad  &&\| S_N^\e\|_{C(\overline{\om})}=O\Big(\e^{2N+2} e^{-\frac{\tht_{min}-\V_\om}{2\e^2}}\Big),
\end{aligned}
\end{equation}
for each subdomain $\om\subset\Om$ and
\begin{align}\label{2.18a}
&
\begin{aligned}
&\|S_N^\e\|_{L_2(\Om)}=O(\e^{2N+3}),\hphantom{ \|\nabla S_N^\e\|_{L_2(\Om)}=O.|\e} &&\quad \|\nabla S_N^\e\|_{L_2(\Om)}=O(\e^{2N+1}), \\ &\|\p^2_{xx}S_N^\e\|_{L_2(\Om)} = O(\e^{2N-1}), && \quad \|S_N^\e\|_{C(\overline{\Om})}=O(\e^{2N+2}).
\end{aligned}
\\
&\label{2.18b}
\begin{aligned}
&\|e^{-\frac{V}{2\e^2}}S_N^\e\|_{L_2(\om)}=O\Big(\e^{2N+3}e^{-\frac{\tht_{min} }{2\e^2}}\Big), && \|\nabla e^{-\frac{V}{2\e^2}} S_N^\e\|_{L_2(\om)}=O\Big(\e^{2N+1}e^{-\frac{\tht_{min} }{2\e^2}}\Big),
\\
&\|\p^2_{xx}  e^{-\frac{V}{2\e^2}} S_N^\e\|_{L_2(\om)} = O\Big(\e^{2N-1}e^{-\frac{\tht_{min} }{2\e^2}}\Big),\qquad  &&\| e^{-\frac{V}{2\e^2}} S_N^\e\|_{C(\overline{\om})}=O\Big(\e^{2N+2}e^{-\frac{\tht_{min} }{2\e^2}}\Big).
\end{aligned}
\end{align}
The symbols $U_j(\z,s)$ denote some polynomials in $\tau$ of degree at most $2j-1$ with infinitely differentiable in $s$ coefficients satisfying the boundary condition $U_j(0,s)=0$. In particular,
\begin{equation}\label{7.16}
\begin{aligned}
U_1=&-\z\tht_1^{-1},
\\
U_2=& -\frac{\z}{18\tht_1^4}\left(2\tht_1^2\z^2-9\tht_1\z  +18 \right) \bigg(3 g^{-1}\nabla_s \tht_0\cdot \nabla_s\tht_1 + 2 \Phi+\Tht_1 \tht_1^{2} + 2\tht_2\tht_1\bigg)
\\
&
+\frac{\Tht_1\z^2}{9\tht_1^2}(\z+9\tht_1) - \frac{\z^2}{18\tht_1^3} (8\tht_2-\Phi\tht_1).
\end{aligned}
\end{equation}
The constant $K_\e$ is given by the formula
\begin{equation}\label{7.2}
K_\e=\frac{\int\limits_{\Om} e^{-\frac{V}{\e^2}}\psi_\e\di x}{\l_\e \int\limits_{\Om} e^{-\frac{V}{\e^2}}\psi_\e^2\di x}
\end{equation}
and it can be represented as
\begin{equation}\label{6.28}
K_\e=K_\e^{(exp)}+ K_\e^{(pow)},\qquad K_\e^{(exp)}:=\e^{-2} \frac{\int\limits_{\Om} e^{-\frac{V}{\e^2}}\di x}{ \int\limits_{\p\Om} e^{-\frac{\tht}{\e^2}}\frac{\p\psi_\e}{\p\tau}\di s}.
\end{equation}
The function $K_\e^{(exp)}$  possesses the asymptotic expansions
\begin{equation}\label{6.10}
K_\e^{(exp)}=\frac{\e^{\frac{2n}{k}}}{\mu_0(\e)} \left(  \sum\limits_{j=0}^{k-1} \e^{\frac{2j}{k}}  K_j^{(exp)} + \sum\limits_{j=k}^{\infty} \e^{\frac{2j}{k}}  K_j^{(exp)} \left(\frac{\mu_1(\e)}{\mu_0(\e)},\ldots,\frac{\mu_{\lf\frac{j}{k}\rf}}{\mu_0(\e)}\right)
\right),
\end{equation}
where $K_j^{(exp)}$, $j=0,\ldots,k+1$, are some constants, and $K_j^{(exp)}$, $j\geqslant k+2$, are some polynomials with fixed coefficients. In particular,
\begin{equation}\label{6.11}
K_0^{(exp)}=\int\limits_{\mathds{R}^n} e^{-V_0}\di x,\qquad K_1^{(exp)}=\int\limits_{\mathds{R}^n} V_1e^{-V_0}\di x.
\end{equation}
The function $K_\e^{(pow)}$ has the following asymptotic expansion:
\begin{equation}\label{6.15}
\begin{aligned}
K_\e^{(pow)}=&-\e^2\frac{\eta_1(\e)}{\mu_0(\e)} + \e^4\left(
\frac{\mu_1(\e)\eta_1(\e)}{\mu_0^2(\e)} - \frac{\eta_2(\e)}{\mu_0(\e)}
\right)
\\
&+ \sum\limits_{j=3}^{\infty} \e^{2j} K_j^{(pow)}
\left(\frac{\eta_1(\e)}{\mu_0(\e)},\ldots, \frac{\eta_{j+1}(\e)}{\mu_0(\e)}, \frac{\mu_1(\e)}{\mu_0(\e)},\ldots, \frac{\mu_j(\e)}{\mu_0(\e)}\right),
\end{aligned}
\end{equation}
where $K_j^{(pow)}$ are some polynomials with fixed coefficients,
\begin{equation}\label{6.13}
\eta_1(\e):=-\int\limits_{\p\Om} e^{-\frac{\tht(s)}{\e^2}}\frac{\di s }{\tht_1(s)},\qquad
\eta_j(\e):=\int\limits_{\p\Om} e^{-\frac{\tht(s)}{\e^2}} \frac{\p U_j}{\p \z}(0,s)\di s,\quad j\geqslant 2,
\end{equation}
and the relations hold:
\begin{equation}\label{6.14}
\eta_j(\e)=O\big(e^{-\frac{\tht_{min}}{\e^2}}\big),\qquad j\geqslant 0.
\end{equation}
\end{theorem}

Let us briefly discuss the main results. Theorems~\ref{thEV},~\ref{thEF} describe the asymptotic expansions of the lowest eigenvalue $\l_\e$ and of the associated eigenfunction. These asymptotics are quite close to the expansions obtained in \cite{HN}. Here we make only a technical extension admitting $V$ to have a degenerate minimum at zero. The main reason why we present these theorems in this section is that the eigenvalue and the eigenfunction are involved in the formulation of the main result in Theorem~\ref{thEq} and from this point of view, Theorems~\ref{thEV},~\ref{thEF} complete Theorem~\ref{thEq}. We also note that the technique used in the proof of Theorems~\ref{thEV},~\ref{thEF}  is in a sense different from the approach in \cite{HN} and our spectral problem is also simpler, since the domain is Euclidean and the function $V$ has the only minimum. An important point is that this technique then works perfectly in the proof of Theorem~\ref{thEq}. This is why, instead of adapting and extending the results and approaches of (a quite lengthy) paper \cite{HN}, we decided to provide a shorter and simpler proof of the asymptotics for the eigenvalue and the eigenfunction; for the latter, we estimate the error term in its asymptotics in a series of various norms.

Asymptotics (\ref{2.9}) for the eigenvalue $\l_\e$  in Theorem~\ref{thEV} is exponentially small, namely, each function $\mu_j$ is exponentially small and its behavior can be described by applying the Laplace method. This method implies   the asymptotics for each $\mu_j$, which is just an exponent $e^{-\frac{\tht_{min}}{\e^2}}$ times some asymptotic series in fractional powers of $\e$ and the final asymptotics for $\l_\e$ can be rewritten as
\begin{equation}\label{2.20}
\l_\e=\e^{-\frac{2n}{k}-2} e^{-\frac{\tht_{min}}{\e^2}} \sum\limits_{j=n-1}^{\infty} C_j \e^{\frac{2j}{p}}
\end{equation}
with some constants $C_j$ and a fixed $p\in\mathds{N}$. The leading term  $M_0$ as well as $\mu_0$ in (\ref{2.9}) are positive since $\tht_1$ is strictly negative by our assumptions.

The asymptotics for the associated function $\Psi_\e$  given in Theorem~\ref{thEF} is  the constant function plus a boundary layer power in $\e$. We provide three types of estimates for the error term. The first series of estimates, (\ref{2.14a}), describes how small the error term is inside $\Om$. As we see, it is exponentially small and the exponent is in fact determined by the maximal value of $V$ in $\om$.  These estimates become too rough if $\V_\om$ exceeds $\tht_{min}$. This can happen, for instance, if $\om$ contains points $x$ near the boundary $\p\Om$ obeying  $V(x)>\tht_{min}$. The smallness of $\Xi_N^\e$ at such points is ensured by the estimates in (\ref{2.15b}). These are global estimates for $\Xi_N^\e$ in the entire domain. They are no longer exponential but power in $\e$. Nevertheless, such estimates are also appropriate since asymptotics (\ref{2.10}) is power in $\e$. The third series, estimates (\ref{2.15c}), are again global estimates but with  a special weight. These estimates are again exponentially small.

The above described structure of the asymptotics of the eigenfunction becomes important once we look at the main result, Theorem~\ref{thEq}. It provides a complete asymptotic expansion for the solution of equation (\ref{3.2}). Its leading term is $K_\e\psi_\e$, where $K_\e$ is an explicitly calculated constant given by (\ref{7.2}). This leading term can be regarded as an external expansion, since the second term in (\ref{2.13}) is just a power in $\e$ boundary layer. For the error term we provide a series of estimates in various norms both on entire domain $\Om$ (see (\ref{2.18a})) and on compact subdomains (see (\ref{2.14})). These estimates are of same nature as the estimates for the error term in the asymptotics for the eigenfunction. The constant $K_\e$ consists of an exponentially large term $K_\e^{(exp)}$ and a power term $K_\e^{(pow)}$, see (\ref{2.18a}). The exponential term is given explicitly in (\ref{2.18a}) and we find its complete asymptotic expansion, see (\ref{6.10}). The behavior of $\mu_j$ has already been discussed above, so, after application of the Laplace method,  asymptotics (\ref{6.10}) can be also rewritten as
\begin{equation*}
  K_\e^{(exp)}=\e^{\frac{2n}{k}-\frac{2(n-1)}{p}}\e^{\frac{\tht_{min}}{2\e^2}} \sum\limits_{j=0}^{\infty}  C_j \e^{\frac{2j}{p}}
\end{equation*}
with some constants $C_j$, where $p$ is the same as in (\ref{2.20}). Here the constant $p$ is the same as in (\ref{2.20}) and the coefficients $C_j$ are of course different. For the power part of $K_\e$, the constant $K_\e^{(pow)}$, we provide a complete asymptotic expansion (\ref{6.15}), which   becomes power after applying of the Laplace
method:
\begin{equation*}
  K_\e^{(pow)}= \e^2\sum\limits_{j=0}^{\infty} C_j \e^{\frac{2j}{p}},
\end{equation*}
where  $p$ is the same as in (\ref{2.20}) and $C_j$ are some coefficients. In view of such structure of $K_\e$, we can state that the solution $u_\e$ involves an exponentially large term, which is $K_\e^{(exp)}\Psi_\e$. This term is given in a closed form in view of the formula for $K_\e^{(exp)}$ in  (\ref{2.18a}) and is characterized by expansions (\ref{2.10}), (\ref{6.10}). The rest of the solution is a sum of $K_\e^{(pow)}\Psi_\e$ and the boundary layer in (\ref{2.13}); for all these components we have complete asymptotic expansions.

In the proofs of Theorems~\ref{thEV},~\ref{thEF},~\ref{thEq} we provide a straightforward recurrent algorithm for determining all coefficients in the above discussed asymptotics. In particular,
the coefficients in asymptotics (\ref{2.9}) arise as the coefficients in the formal asymptotics for the quotient:
\begin{equation*}
\l_\e\sim
\frac{\sum\limits_{j=0}^{\infty}\e^{2j-2}\mu_j(\e)} {\sum\limits_{j=n}^{\infty}\e^{\frac{2j}{k}}\a_k},\qquad \int\limits_{\Om} e^{-\frac{V}{\e^2}}\di x\sim \sum\limits_{j=n}^{\infty}\e^{\frac{2j}{k}}\a_k,
\end{equation*}
where the series in denominator arises by applying the Laplace method to the mentioned integral. Similarly, the coefficients in asymptotics (\ref{6.10}), (\ref{6.15})   come from the formal asymptotics of the quotients
\begin{equation*}
K_\e^{(exp)}\sim \frac{\sum\limits_{j=n}^{\infty}\e^{\frac{2j}{k}}\a_k} {\sum\limits_{j=0}^{\infty}\e^{2j}\mu_j(\e)},\qquad
K_\e^{(pow)}\sim -\frac{\sum\limits_{j=1}^{\infty}\e^{2j-2}\eta_j(\e)} {\sum\limits_{j=0}^{\infty}\e^{2j-2}\mu_j(\e)}.
\end{equation*}
The boundary layer in (\ref{2.10}) is constructed as a formal asymptotic solution of the problem
\begin{equation*}
(-\e^2\D+\nabla V\cdot\nabla)e^{\frac{\tht_1(s)\tau}{\e^2}} \sum\limits_{j=0}^{\infty}\e^{2j} \Phi_j(\tau\e^{-2},s)=0,\qquad \sum\limits_{j=0}^{\infty}\e^{2j} \Phi_j(\tau\e^{-2},s)=1,
\end{equation*}
and a similar problem for the boundary layer in (\ref{2.13}) reads as
\begin{equation*}
(-\e^2\D+\nabla V\cdot\nabla)e^{ \frac{\tht_1(s)\tau}{\e^2}} \sum\limits_{j=1}^{\infty}\e^{2j} U_j(\tau\e^{-2},s) =e^{\frac{\tht_1(s)\tau}{\e^2}} \sum\limits_{j=0}^{\infty}\e^{2j} \Phi_j(\tau\e^{-2},s),\qquad \sum\limits_{j=1}^{\infty}\e^{2j} U_j(\tau\e^{-2},s)=0.
\end{equation*}

\section{Probabilistic interpretation}

In this section we discuss probabilistic aspects of the results in Theorems~\ref{thEV},~\ref{thEF},~\ref{thEq}. We recall that in our case $F=-\nabla V$. The mean exit time $u_\e$ is a nonnegative function and this is reflected by the structure of asymptotics (\ref{2.13}). Indeed, inside $\Om$, the eigenfunction $\Psi_\e$ is approximately $1$, while the constant $K_\e$ is positive (see formulae (\ref{6.28}), (\ref{6.10}), (\ref{6.11})) and is exponentially large. In the vicinity of $\p\Om$, the leading terms in the asymptotics for $\Psi_\e$ are $1-e^{\frac{\tht_1(s)\tau}{\e^2}}$, while the leading term of the boundary layer in (\ref{2.13}) is $-\tau\tht_1^{-1}e^{\frac{\tht_1(s)\tau}{\e^2}}$ and thanks to the presence of $K_\e$ at $\Psi_\e$, the total sum in (\ref{2.13}) is again positive. The mean exit time is exponentially large in $\e$ as it follows from  the asymptotics for $u_\e$. We can describe various properties of the exit time by studying the structure of asymptotics (\ref{2.13}).

First we discuss two examples of such properties, namely, the maximal exit time and the torsional rigidity. The former is defined as $\max\limits_{\overline{\Om}} u_\e(x)$, while the latter is the integral $\int\limits_{\Om} u_\e(x)\di x$.

We let $N=0$ in (\ref{2.10}), (\ref{2.13}) and by (\ref{2.12}), (\ref{2.15b}), (\ref{2.18a}) we obtain
\begin{equation}\label{2.21}
\big|\Psi_\e(x)-1+\chi(x)e^{\frac{\tht_1(s)\tau}{\e^2}}\big|=O(\e^2),\qquad |u_\e(x)-K_\e\Psi_\e(x)|=O(\e^2)
\end{equation}
uniformly in $\overline{\Om}$. By (\ref{2.6a}) and the definition of $\tht_1$ we also get
\begin{equation*}
\big|\chi(x)e^{\frac{\tht_1(s)\tau}{\e^2}}\big|=O(\e^2)\quad\text{as}\quad \tau\geqslant \vk(\e):=2c_2^{-1}\e^2|\ln\e|.
\end{equation*}
Hence,  the maximum of $u_\e$ is attained inside $\Om_\e:=\Om\setminus\Pi_{\vk(\e)}$, $\Pi_\vk:=\{x:\, 0<\tau<\vk\}$ and this maximum is equal to $K_\e+O(\e^2)$. This maximum coincides with the values of $u_\e$ at all other points in $\Om_\e$ up to the error $O(\e^2)$. Hence, the exit time from all points of $\Om$
outside a narrow layer $\Pi_{\vk(\e)}$ along $\p\Om$ is approximately the same and is equal to $K_\e+O(\e^2)$.

Since the boundary layers in (\ref{2.10}), (\ref{2.13})  is multiplied by the cut-off function, it is absent outside $\Pi_\d$, see the definition of $\chi$. By (\ref{2.6a}) we also have bound $V(x)\leqslant \tht_{min}-c_2\d$ outside $\Pi_\d$ and thanks to (\ref{2.14a}), (\ref{2.14}) we can improve (\ref{2.21}):
\begin{equation*}
\big|\Psi_\e(x)-1\big|= O\big(\e^N e^{-\frac{c_2\d}{\e^2}}\big),\qquad |u_\e(x)-K_\e\Psi_\e(x)|=O\big(\e^Ne^{-\frac{c_2\d}{\e^2}}\big)
\end{equation*}
uniformly in $\overline{\Om\setminus\Pi_\d}$ for each natural $N$. Hence,
\begin{equation*}
u_\e(x)=K_\e\Psi_\e(x) + O\big(\e^N e^{-\frac{c_2\d}{\e^2}}\big)= K_\e+O\big(\e^N e^{\frac{\tht_{min}-c_2\d}{\e^2}}\big)
\end{equation*}
uniformly in $\overline{\Om\setminus\Pi_\d}$, where the asymptotic behavior of $K_\e$ is due to (\ref{6.28}), (\ref{6.10}), (\ref{6.15}). The latter asymptotic formula states that inside $\overline{\Om\setminus\Pi_\d}$, the exit time is approximated by $K_\e\Psi_\e$ up to an exponentially small error; we observe that at the same time, $K_\e$ is exponentially large. Once we replace $\Psi_\e$ by $1$, the final error becomes worse due to the exponential growth of $K_\e$ but it is still exponentially small in comparison with $K_\e$.

We proceed to the torsional rigidity. First we integrate
the terms of the boundary layers in asymptotics (\ref{2.10}) employing (\ref{2.6a}) and the identities $x=\x(s)+\tau \nu(s)$, $dx=\det(I-\tau b(s))\di \tau \di s$. This gives:
\begin{align*}
\int\limits_{\Om} \e^{2j}\chi(x) e^{\frac{\tht_1(s)\tau}{\e^2}} \Phi_j(\tau\e^{-2},s)\di x=&\e^{2j} \int\limits_{\p\Om\times(0,\d)} \chi\big(\x(s)+\tau\nu(s)\big) \det\big(I-\tau b(s)\big)\Phi_j(\tau\e^{-2},s)  e^{\frac{\tht_1(s)\tau}{\e^2}}  \di\tau\di s
\\
=&\e^{2j+2} \int\limits_{0}^{\d\e^{-2}}\di \z \int\limits_{\p\Om} \chi\big(\x(s)+\e^2\z\nu(s)\big) \det\big(I-\e^2\z b(s)\big)\Phi_j(\z,s)e^{\tht_1(s)\z}\di\z\di s
\\
=&\e^{2j+2} \int\limits_{0}^{+\infty}\di \z \int\limits_{\p\Om} \det\big(I-\e^2\z b(s)\big) \Phi_j(\z,s) e^{\tht_1(s)\z}\di\z\di s+O(\e^2e^{-\frac{c_2\d}{3}\e}).
\end{align*}
Together with (\ref{2.10}), (\ref{2.15b}) this implies:
\begin{equation*}
\int\limits_{\Om} \Psi_\e(x)\di x=|\Om|-\sum\limits_{j=0}^{N}\e^{2j+2} \int\limits_{0}^{+\infty}\di \z \int\limits_{\p\Om} \det\big(I-\e^2\z b(s)\big) \Phi_j(\z,s)e^{\tht_1(s)\z}\di\z\di s + O(\e^{2N+3}).
\end{equation*}
In the same way we integrate the boundary layer in (\ref{2.13}) arriving at the final asymptotics for the torsional rigidity:
\begin{align*}
\int\limits_{\Om} u_\e(x)\di x=&K_\e\int\limits_{\Om}\Psi_\e(x)\di x + \sum\limits_{j=1}^{N} \e^{2j+2} \int\limits_{0}^{+\infty}\di \z \int\limits_{\p\Om} \det\big(I-\e^2\z b(s)\big) U_j(\z,s)e^{\tht_1(s)\z}\di\z\di s +O(\e^{2N+3})
\\
=& K_\e\left(|\Om| -  \sum\limits_{j=0}^{N}\e^{2j+2} \int\limits_{0}^{+\infty}\di \z \int\limits_{\p\Om} \det\big(I-\e^2\z b(s)\big) \Phi_j(\z,s)e^{\tht_1(s)\z}\di\z\di s\right) + O\big(\e^{2N+3+c_3}e^{\frac{\tht_{min}}{2\e^2}}\big)
\end{align*}
with some fixed $c_3>0$. The determinant in the obtained formula can be expanded into a polynomial in $\e^2\z$ of degree $n-1$.

The next characteristics is a quasi-stationary distribution for process (\ref{0.1}). This is a probability measure $\vt_\e$ such that once $\mathbf{Y}_\e(0,a)$ is distributed according $\vt_\e$, the identity
\begin{equation}\label{2.24}
\mathbb{P}\big(\mathbf{Y}_\e(t,a)\in\om:\, t<T_\e(a)\big)=\vt_\e(\om)
\end{equation}
holds for all $t$ and all Borel subsets $\om\subset \Om$, see \cite{Replica}, \cite[Eq. 6]{Nuh1}, \cite[Prop. 5]{Nuh2} and the references therein. There exists the unique quasi-distribution for process (\ref{0.1}) and it is given by
\begin{equation}\label{2.23}
d\vt_\e=\frac{e^{-\frac{V(x)}{\e^2}}\Psi_\e(x)}{\int\limits_\Om e^{-\frac{V(x)}{\e^2}}\Psi_\e(x)\di x}\di x.
\end{equation}
Theorem~\ref{thEF} allows us to describe the asymptotic behavior of this measure. Namely,  the asymptotics of the numerator in (\ref{2.23}) reads as
\begin{equation}\label{2.26}
e^{-\frac{V(x)}{\e^2}}\Psi_\e(x)=e^{-\frac{V(x)}{2\e^2}} \left( e^{-\frac{V(x)}{2\e^2}} -\chi(x)e^{-\frac{V(x)-2\tht_1(s)\tau}{2\e^2}} \sum\limits_{j=0}^{N}\e^{2j} \Phi_j(\tau\e^{-2},s)  +
e^{-\frac{V(x)}{2\e^2}}
\Xi_{\e,N}(x)\right),
\end{equation}
where the error term inside the brackets is to be estimated by (\ref{2.15c}). The asymptotics for the denominator is as follows:
\begin{equation}\label{2.25}
\int\limits_{\Om} e^{-\frac{V(x)}{\e^2}}\Psi_\e(x)\di x=\int\limits_{\Om} e^{-\frac{V(x)}{\e^2}}\di x + O\big(\e e^{-\frac{\tht_{min}}{2\e^2}}\big).
\end{equation}
As an example of application of the obtained formulae, let us find
the right hand in (\ref{2.24}). We have:
\begin{align}\label{2.27a}
&\int\limits_{\om} e^{-\frac{V(x)}{\e^2}} \Psi_\e(x)\di x=\int\limits_{\om} e^{-\frac{V(x)}{\e^2}}\di x + O\big(\e^N e^{-\frac{\tht_{min}}{2\e^2}}\big)\quad \text{for all}\quad N\in\mathds{N} && \text{if}\quad \dist(\om,\p\Om)>\tfrac{2\d}{3},
\\
&\label{2.27b}
\int\limits_{\om} e^{-\frac{V(x)}{\e^2}} \Psi_\e(x)\di x=\int\limits_{\om} e^{-\frac{V(x)}{\e^2}}\di x + O\big( e^{-\frac{\tht_{min}-c_2\dist(\om,\p\Om)}{2\e^2}}\big) && \text{if}\quad 0<\dist(\om,\p\Om)<\tfrac{2\d}{3},
\\
&\label{2.27c}
\int\limits_{\om} e^{-\frac{V(x)}{\e^2}} \Psi_\e(x)\di x=\int\limits_{\om} e^{-\frac{V(x)}{\e^2}}\di x + O\big( \e^{2-\frac{2n}{k}+\frac{n-1}{p}} e^{-\frac{1}{2\e^2}(\tht_{min}-2\min_{\overline{\Pi_\d}}V)}\big) && \text{if} \quad \dist(\om,\p\Om)=0;
\end{align}
here $p$ is the same as in (\ref{2.20}). The latter estimate is obtained by applying the technique from the proof of Lemma~\ref{lm7.3} in Section~7. Three last formulae and (\ref{2.24}), (\ref{2.23}), (\ref{2.25}) yield:
\begin{equation*}
\mathbb{P}\big(\mathbf{Y}_\e(t,a)\in\om:\, t<T_\e(a)\big)= \frac{\int\limits_{\om} e^{-\frac{V(x)}{\e^2}}\di x + O\big(\e^{c_4}e^{-\frac{c_5}{\e^2}}\big)}{\int\limits_{\Om} e^{-\frac{V(x)}{\e^2}}\di x + O\big(\e e^{-\frac{\tht_{min}}{2\e^2}}\big)},
\end{equation*}
where $c_4$, $c_5$ are some fixed positive constants and $c_5\geqslant \tfrac{\tht_{min}}{2}-\min\limits_{\overline{\Pi_\d}} V$. The asymptotics for the integrals in the obtained formula can be found by the Laplace method.

The quasi-stationary distribution is closely related with the law of $\mathbf{Y}_\e(T_\e(a),a)$ over the boundary $\p\Om$, where, we recall, $T_\e(a)$ is the exit time introduced in (\ref{0.7}). This law characterizes the probability of the first hitting the points on $\p\Om$ by the trajectories. According \cite[Prop. 6]{Nuh2}, the density of this law is given by
\begin{equation*}
d\mathbf{Y}_\e(T_\e(a),a)=\frac{\e^2}{\l_\e} \frac{ \frac{\p\ }{\p\tau} \Psi_\e e^{-\frac{V}{\e^2}}
}{\int\limits_{\Om}e^{-\frac{V}{\e^2}}\Psi_\e\di x}\di s\qquad \text{on}\quad\p\Om.
\end{equation*}
  Substituting formula (\ref{5.14a}) in the above density, we rewrite it as
\begin{equation*}
d\mathbf{Y}_\e(T_\e(a),a)=  \frac{e^{-\frac{V}{\e^2}} \frac{\p \Psi_\e }{\p\tau}
}{\int\limits_{\p\Om}e^{-\frac{V}{\e^2}} \frac{\p \Psi_\e }{\p\tau}\di s}\di s\qquad \text{on}\quad\p\Om.
\end{equation*}
Thanks to the estimates (\ref{2.15b}), asymptotics (\ref{2.10}) holds in $\H^2(\Om)$-norm and by the standard embedding theorem, we can differentiate (\ref{2.10}) with respect to $\tau$ at the boundary $\p\Om$. This gives the following asymptotics in the sense of $L_2(\p\Om)$-norm:
\begin{equation*}
\frac{\p\Psi_\e}{\p\tau}\Bigg|_{\p\Om}=-\frac{\p\ }{\p\tau} e^{\frac{\tht_1(s)\tau}{\e^2}} \sum\limits_{j=0}^{N}\e^{2j} \Phi_j(\tau\e^{-2},s) +O(\e^{2N-1})=-\e^{-2}\tht_1(s) - \sum\limits_{j=0}^{N}\e^{2j}\frac{\p\Phi_j}{\p\z}(0,s) + O(\e^{2N-1}).
\end{equation*}
Here we have also employed formula (\ref{2.12}) and the boundary condition $\Phi_j(0,s)=0$, $j\geqslant 1$. The obtained asymptotics implies:
\begin{equation}\label{2.29}
 \frac{e^{-\frac{V(x)}{\e^2}} \frac{\p \Psi_\e }{\p\tau}
}{\int\limits_{\p\Om}e^{-\frac{V}{\e^2}} \frac{\p \Psi_\e }{\p\tau}\di s}= -e^{-\frac{\tht_0}{\e^2}} \frac{  \e^{-2}\tht_1(s) + \sum\limits_{j=1}^{N} \e^{2j-2} \frac{\p\Phi_j}{\p\z}(0,s) + O(\e^{2N-1}) }{\sum\limits_{j=0}^{N} \e^{2j-2}\mu_j(\e) + O(\e^{2N-1}\e^{-\frac{\tht_{min}}{\e^2}})},\qquad x\in\p\Om,
\end{equation}
with $\mu_j$ defined in (\ref{5.23}). Dividing the asymptotic series in the right hand side, one can find easily the complete asymptotic expansion for the considered density. This density allows also to find the  expectation value $\mathbb{E}_{\vt_\e}\big(f(\mathbf{Y}_\e(T_\e(a),a))\big)$ with respect to the measure $\vt_\e$ for each $f\in L_\infty(\p\Om)$. Namely, by \cite[Eq. (15)]{Nuh1}, if $\mathbf{Y}_\e(T_\e(a),a)$ is distributed according the measure $\vt_\e$, then
\begin{equation*}
\mathbb{E}_{\vt_\e}\big(f(\mathbf{Y}_\e(T_\e(a),a))\big) = \frac{\int\limits_{\p\Om} f e^{-\frac{V}{\e^2}} \frac{\p \Psi_\e }{\p\tau}\di s
}{\int\limits_{\p\Om}e^{-\frac{V}{\e^2}} \frac{\p \Psi_\e }{\p\tau}\di s}
\end{equation*}
and by (\ref{2.29}),
\begin{equation*}
\mathbb{E}_{\vt_\e}\big(f(\mathbf{Y}_\e(T_\e(a),a))\big) = - \frac{\e^{-2}\int\limits_{\p\Om} f(s) e^{-\frac{\tht_0(s)}{\e^2}}\di s +
\sum\limits_{j=0}^{N} \e^{2j-2} \int\limits_{\p\Om} f(s) \frac{\p\Phi_j}{\p\z}(0,s) e^{-\frac{\tht_0(s)}{\e^2}}\di s
+O\big(\e^{2N-1}e^{-\frac{\tht_{min}}{\e^2}}\big)}{\sum\limits_{j=0}^{N} \e^{2j-2}\mu_j(\e) + O(\e^{2N-1}\e^{-\frac{\tht_{min}}{\e^2}})}.
\end{equation*}
Dividing the asymptotics series in the obtained quotient, one can find easily the complete asymptotic expansion for the considered expectation value.

\section{Preliminaries}

In this section we prove certain auxiliary facts, which will be employed later in proofs of Theorems~\ref{thEV},~\ref{thEF},~\ref{thEq}.
We introduce one more operator in $L_2(\Om)$ with the differential expression
\begin{equation}\label{3.3}
\Op^\e:=-\e^2\D+W_\e,\qquad W_\e(x):=\frac{1}{4\e^2}|\nabla V(x)|^2-\frac{1}{2}\D V(x),
\end{equation}
on the domain $\Dom(\Op^\e):=\Ho^2(\Om)$. This operator is self-adjoint. The associated quadratic form is defined by the formula
\begin{equation}\label{3.10}
\hf^\e(u):=\e^2\|\nabla u\|_{L_2(\Om)} + (W_\e u,u)_{L_2(\Om)}
\end{equation}
on the domain $\Dom(\hf^\e):=\Ho^1(\Om)$. It is straightforward to check that the operators $\Op_\e$ and $\Op^\e$ are related by the identity
\begin{equation}\label{3.4}
\Op^\e=\cE_\e^{-1} \Op_\e \cE_\e,
\end{equation}
where $\cE_\e$ is the operator of multiplication by $e^{\frac{V}{2\e^2}}$, that is, $\cE_\e u=e^{\frac{V}{2\e^2}} u$.
Denoting
\begin{equation}\label{3.5}
u^\e:=\cE_\e^{-1} u_\e,\qquad u_\e=\cE_\e u^\e,
\end{equation}
we rewrite equation (\ref{3.2}) as
\begin{equation}\label{3.6}
\Op^\e u^\e=\psi_\infty^\e,\qquad \psi_\infty^\e:=\cE_\e^{-1}\mathds{1},\qquad \psi_\infty^\e(x) =e^{-\frac{V(x)}{2\e^2}}.
\end{equation}
Exactly this equation will be studied in the proof of Theorem~\ref{thEq}.

Since the domain $\Om$ is bounded, the operator $\Op^\e$ has a compact resolvent and its spectrum consists of infinitely many discrete eigenvalues accumulating at infinity. Thanks to
(\ref{3.4}), the spectra of the operator $\Op_\e$ and $\Op^\e$ coincide and this proves the first statement in Theorem~\ref{thEV}.
We arrange the eigenvalues of $\Op^\e$ in the ascending order counting the multiplicities and we denote them by $\l_\e:=\l_1^\e<\l_2^\e\leqslant \l_3^\e\leqslant \ldots$

The unique solvability of equation (\ref{3.6}) is equivalent to the fact that zero is not among the eigenvalues $\l_j^\e$. This is indeed the case as we shall show in the proof of the following statement.

\begin{lemma}\label{lm3.1}
The estimates
\begin{equation*}
0<\l_\e<c_6\e^{-\frac{2n}{k}}e^{-\frac{\tht_{min}-c_2\d}{\e^2}},
\qquad \l_2^\e>c_7\e^{2-\frac{4}{k}}
\end{equation*}
hold, where $c_6$, $c_7$ are positive constants independent of $\e$.
\end{lemma}

\begin{proof}
Thanks to assumption (\ref{2.1a}), we can continue the function $V$ outside  $\Om$ so that the continuation is infinitely differentiable outside $\Om$ and also $V(x)=|x|^2$ as $|x|\geqslant \rho_2>0$
for sufficiently large $\rho_2$. Then we can also continue the potential $W_\e(x)$ outside $\Om$ defining it again by the formula in (\ref{3.3}).

By $\Op_\infty^\e$ we denote the self-adjoint operator associated with the sesquilinear form
\begin{equation*}
\hf_\infty^\e(v):=\e^2\|\nabla v\|_{L_2(\mathds{R}^n)} + (W_\e v,v)_{L_2(\mathds{R}^n)}
\end{equation*}
on the domain $\Dom(\hf_\infty^\e):=\H^1(\mathds{R}^n)\cap L_2\big(\mathds{R}^n,(1+|x|^2)\di x\big)$. This    has a compact resolvent  \cite[Ch. X\!I\!I\!I, Sect. 14, Thm. X\!I\!I\!I.67]{RS} and its spectrum   consists of discrete eigenvalues   $\L_j^\e$, $j\geqslant 1$, which we arrange in the ascending order counting the multiplicities: $\L_1^\e<\L_2^\e\leqslant \L_3^\e\leqslant\ldots$
Given a function $v\in\Dom(\hf_\Om^\e)$,  we extend it by zero outside $\Om$ and the extension belongs to $\Dom(\hf_\infty^\e)$. In view of the definition of the forms $\hf_\Om^\e$ and $\hf_\rho^\e$, we see that $\hf^\e\geqslant \hf_\infty^\e$ and by the minimax principle
\begin{equation}\label{3.9}
\L_1^\e\leqslant \l_\e,\qquad \L_2^\e\leqslant \l_2^\e.
\end{equation}

It is straightforward to check that the function
$\psi_\infty^\e$ introduced in (\ref{3.6}) with the continued potential $V$
is the eigenfunction of the operator $\Op_\infty^\e$ associated with the zero eigenvalue. This eigenfunction is positive and hence, this is the ground state and $\L_1^\e=0$. By (\ref{3.9}) this implies
that $\l_\e$ is non-negative. The eigenvalue $\l_\e$ is also non-zero. Indeed, otherwise  we continue the associated eigenfunction by zero outside   $\Om$ and this function provides the infimum of the quotient $\hf_\infty^\e(u)/\|u\|_{L_2(\mathds{R})}^2$. Then this function is also the ground state of the operator $\Op_\infty^\e$ and since it vanishes outside $\Om$,  this  contradicts the unique continuation principle. This proves the lower bound for $\l_\e$.

To prove the upper bound for $\l_\e$, we again employ the minimax principle implying that
\begin{equation}\label{3.11}
\l_\e\leqslant \frac{\hf^\e\big(\cE^{-1}(1-\chi)\big)}{\|\cE^{-1}(1-\chi)\|_{L_2(\Om)}^2}.
\end{equation}
It follows from assumptions (\ref{2.1a}), (\ref{2.6a}) that
\begin{equation}\label{3.14}
V(x)\geqslant  C|x|^k\quad\text{as}\quad |x|\leqslant \rho_1,\qquad
V(x)\leqslant \tht(s)-c_2\tau\quad \text{as}\quad 0\leqslant \tau\leqslant \d,
\end{equation}
where $C$ is some fixed positive constant independent of $x$.
Employing identity (\ref{3.4}) and  definition (\ref{3.10}), (\ref{3.13a}) of the form $\hf^\e$ and the cut-off function $\chi$,   we straightforwardly check
\begin{align*}
\hf^\e\big(\cE^{-1}(1-\chi)\big)= \big(\cE_\e^{-1}\Op_\e(1-\chi),\cE^{-1}(1-\chi)\big)_{L_2(\Om)}
=\int\limits_{\Pi_{\tau_0}} e^{-\frac{V(x)}{\e^2}}(\e^2\D\chi(x)-\nabla V(x)\cdot\nabla\chi (x))\big(1-\chi(x)\big)\di x,
\end{align*}
$\Pi_{\tau_0}=\{x:\,0<\tau<\tau_0\}$, and by   (\ref{2.1a}), (\ref{3.14})  we get the estimates
\begin{equation*}
\hf^\e\big(\cE^{-1}(1-\chi)\big)=O\Big(e^{-\frac{\tht_{min}-c_2\d }{\e^2}}\Big),\qquad \| \cE_\e^{-1}(1-\chi)\|_{L_2(\Om)}^2 \geqslant\int\limits_{B_{\rho_1}} e^{-\frac{C|x|^k}{\e^2}}\di x
=\e^\frac{2n}{k}  \int\limits_{0}^{\rho_1\e^{-\frac{2}{k}}} e^{-C|y|^k}\di y\geqslant C\e^{\frac{2n}{k}},
\end{equation*}
where $C$ denotes inessential positive constants independent of $\e$ and $x$. These estimate  and (\ref{3.11}) lead us to the needed upper bound for $\l_\e$.

We proceed to proving the  lower bound for the second eigenvalue. We first consider the case $k=2$, and we apply the result of \cite[Thm. 3.2]{QC} to the operator $\Op_\infty^\e$. In view of (\ref{2.1a}), (\ref{2.6a})
we immediately get that
\begin{equation}\label{3.13}
\liminf\limits_{\e\to+0} \e^{2-\frac{4}{k}}\L_2^\e=\L_2^0,
\end{equation}
where $\L_2^0$ is the second eigenvalue of a self-adjoint operator
$\Op_0$ associated with the quadratic form
\begin{equation*}
\hf_0(u):=\|\nabla u\|_{L_2(\mathds{R}^n)}^2 + (W_0 u,u)_{L_2(\mathds{R}^n)},
\qquad  W_0(x):=\frac{|\nabla V_0(x)|^2}{4}-\frac{\D V_0(x)}{2},
\end{equation*}
in $L_2\big(\mathds{R}^n,(1+|x|^{2k-2})\di x\big)$ on the domain $\Dom(\hf_0):=\H^1(\mathds{R}^n)\cap L_2(\mathds{R}^n,(1+|x|^{2k-2})\di x)$. In case $k>2$ the same result can be proved as above by an appropriate scaling in the definition of the operator $H_a$ in \cite[Eq. (7)]{QC}.

It is easy to see that zero is the eigenvalue of the operator $\Op_0$ and the associated eigenfunction is $e^{-\frac{V_0(x)}{2}}$. Since this eigenfunction is positive, it is the ground state and zero is the lowest eigenvalue. Hence, $\L_2^0>0$ and by (\ref{3.13}) we arrive at the needed lower bound for $\l_2^\e$.
\end{proof}

The following  auxiliary lemma is one of the main tools in
 the proofs of Theorems~\ref{thEV},~\ref{thEF},~\ref{thEq}.

\begin{lemma}\label{lm:Max}
For each function $u\in C^2(\overline{\Om})\cap C^\infty(B_{\rho_1})\cap\Dom(\Op^\e)$
and  each $\l\in[0,\e^2]$ the estimates hold:
\begin{align}
&\label{3.26a}
\|\nabla u\|_{L_2(\Om)}\leqslant C\e^{-2} \left(\|(\Op^\e-\l)u\|_{L_2(\Om)}+\|u\|_{L_2(\Om)}\right),
\\
&\label{3.26b}
\|\p^2_{xx} u\|_{L_2(\Om)}\leqslant C\e^{-4} \left(\|(\Op^\e-\l)u\|_{L_2(\Om)}+\|u\|_{L_2(\Om)}\right),
\\
& \label{6.19}
  \|u\|_{\H^{m+1}(B_{\frac{\rho_1}{2}})}\leqslant
  C\e^{-2m-2} \left(\|(\Op^\e-\l)u\|_{\H^{m-1}(B_{\rho_1})} + \|(\Op^\e-\l)u\|_{L_2(\Om)}
 + \|u\|_{L_2(\Om)}\right),\quad m\in\mathds{N},
\\
\label{3.23a}
& \|u\|_{C(\overline{\Om})} \leqslant C\e^{-n-2}\left(
\|(\Op^\e-\l)u\|_{\H^{\lf\frac{n}{2}\rf-1}(B_{\rho_1})}
+
\|(\Op^\e-\l)u\|_{C(\overline{\Om})}+\|u\|_{L_2(\Om)}
\right),
\\
\label{3.23b}
& \big\|e^{\frac{V}{2\e^2}}u\big\|_{C(\overline{\Om})} \leqslant C\e^{-n-5+\frac{4}{k}} \left(\|(\Op^\e-\l)u\|_{\H^{\lf\frac{n}{2}\rf-1}(B_{\rho_1})}
+\|e^{\frac{V}{2\e^2}}(\Op^\e-\l)u\|_{C(\overline{\Om})} +\|u\|_{L_2(\Om)}\right),
\\
&\label{3.29}
\big\|\nabla e^{\frac{V}{2\e^2}}u\big\|_{L_2(\Om)} \leqslant C\e^{-n-6+\frac{4}{k}} \left(\|(\Op^\e-\l)u\|_{\H^{\lf\frac{n}{2}\rf-1}(B_{\rho_1})}
+\|e^{\frac{V}{2\e^2}}(\Op^\e-\l)u\|_{C(\overline{\Om})}+\|u\|_{L_2(\Om)}\right),
\\
&\label{3.30}
\|\p^2_{xx} e^{\frac{V}{2\e^2}} u\|_{L_2(\Om)}\leqslant C\e^{-n-8+\frac{4}{k}}\left(\|(\Op^\e-\l)u\|_{\H^{\lf\frac{n}{2}\rf-1}(B_{\rho_1})}
+\|e^{\frac{V}{2\e^2}}(\Op^\e-\l)u\|_{C(\overline{\Om})} +\|u\|_{L_2(\Om)}\right),
\end{align}
where $C$ is a constant independent of $u$, $\e$ and $\l$.
\end{lemma}

\begin{proof}
Throughout the proof, the symbol $C$ stands for various inessential constants independent of $u$, $\e$ and $\l$. Estimate  (\ref{3.26a}) is implied immediately  by the integral identity
$\e^2\|\nabla u\|_{L_2(\Om)}^2=(f,u)_{L_2(\Om)}-((W_\e-\l) u,u)_{L_2(\Om)}$ and the definition of $W_\e$. Since $\|\D u\|_{L_2(\Om)}=\e^{-2}\|(\l-W_\e)u+f\|_{L_2(\Om)}$, we can estimate first the  norm $\|\D u\|_{L_2(\Om)}$ and apply then the standard estimates in Sobolev spaces \cite[Ch. I\!I\!I, Sect. 8, Lm. 8.1]{Ld}, namely, $\|\p^2_{xx} u\|_{L_2(\Om)}\leqslant C\big(\|\D u\|_{L_2(\Om)}+\|u\|_{L_2(\Om)}\big)$. This gives (\ref{3.26b}).

Let $\chi_1=\chi_1(t)$ be an infinitely differentiable cut-off function equalling to one as $|x|<\rho_2$ and vanishing as $|x|>\rho_3$, where $\rho_2$, $\rho_3$ are some numbers obeying $\frac{\rho_1}{2}<\rho_2<\rho_3<\rho_1$. The function $u\chi_1$ solves the boundary value problem
\begin{equation}\label{3.21}
-\D u\chi_1 =f_\chi:=\e^{-2}(f-(W_\e-\l) u)\chi_1 -2\nabla u\cdot \nabla \chi_1 -  u\Delta\chi_1\quad \text{in}\quad B_{\rho_1}, \qquad u\chi_1=0\quad\text{on}\quad \p B_{\rho_1}.
\end{equation}
We differentiate this equation with respect to $x_i$, $i=1,\ldots,n$:
\begin{equation*}
-\D \frac{\p u\chi_1}{\p x_i} =\frac{\p f_{\chi}}{\p x_i},\quad \text{in}\quad B_{\rho_1}, \qquad \frac{\p u\chi_1}{\p x_i}=0\quad\text{on}\quad \p B_{\rho_1}.
\end{equation*}
Estimates (\ref{3.26a}), (\ref{3.26b}) imply:
\begin{equation*}
\|f_\chi\|_{\H^1(B_{\rho_1})} \leqslant C\left(
\e^{-2} \|f\|_{\H^1(B_{\rho_1})} + \e^{-6}\big(\|f\|_{L_2(\Om)} + \|u\|_{L_2(\Om)}\big)
\right).
\end{equation*}
Employing then  \cite[Ch. I\!I\!I, Sect. 8, Lm. 8.1]{Ld} once again, we get:
\begin{equation*}
\|u\|_{\H^3(B_{\rho_2})}\leqslant \|u\chi_1\|_{\H^2(B_{\rho_3})}\leqslant C\|f_{\chi}\|_{\H^1(B_{\rho_1})}\leqslant C\e^{-6} \left(\|f\|_{\H^1(B_{\rho_1})} + \|f\|_{L_2(\Om)}
 + \|u\|_{L_2(\Om)}\right).
\end{equation*}
Now we replace $\rho_2$, $\rho_3$ by another pair $\tilde{\rho}_2$, $\tilde{\rho}_3:=\rho_2$ and repeat the above arguing but differentiating the equation in (\ref{3.21}) twice. This gives the estimate
\begin{equation*}
\|u\|_{\H^4(B_{\tilde{\rho}_3})}\leqslant C\e^{-8}
\left(\|f\|_{\H^2(B_{\rho_1})} + \|f\|_{L_2(\Om)}
 + \|u\|_{L_2(\Om)}\right).
\end{equation*}
Repeating this procedure as many times as needed, we finally obtain (\ref{6.19}). Thanks to  the embedding $\H^{\lf\frac{n}{2}\rf+1}(B_{\frac{\rho_1}{2}})\subset C(\overline{B_{\frac{\rho_1}{2}}})$, estimate (\ref{6.19}) with $m=\lf\frac{n}{2}\rf$ implies
\begin{equation}\label{3.33}
\|u\|_{C(\overline{B_{\frac{\rho_1}{2}}})} \leqslant C\e^{-n-2}
\left(\|f\|_{\H^{\lf\frac{n}{2}\rf-1}(B_{\rho_1})} + \|f\|_{L_2(\Om)}
 + \|u\|_{L_2(\Om)}\right).
\end{equation}
By assumptions (\ref{2.1a}), (\ref{2.6a}) and Lemma~\ref{lm3.1}, for small $\e$, the potential $W_\e-\l$ is strictly positive outside $B_{\frac{\rho_1}{2}}$, namely, $W_\e-\l\geqslant C>0$ in $\overline{\Om\setminus B_{\frac{\rho_1}{2}}}$. Applying then the classical maximum principle in this domain to the differential equation for $u$ and using (\ref{3.33}), we arrive at (\ref{3.23a}).

It also follows from (\ref{2.1a}), (\ref{2.6a}) that there exists a fixed constant $\rho_4>0$ such that
\begin{equation}\label{3.36}
\begin{aligned}
&e^{\frac{V(x)}{2\e^2}}\leqslant \e^{-1}\quad \text{in}\quad  B_*:=\Big\{x: |x|<\rho_4\e^{\frac{2}{k}}|\ln\e|^\frac{1}{k}\Big\},
\\
&\e^{-2}|\nabla V(x)|^2 \geqslant C\e^{2-\frac{4}{k}}|\ln \e|^{2-\frac{2}{k}},\qquad |\D V|\leqslant   C\e^{2-\frac{4}{k}}|\ln \e|^{1-\frac{2}{k}}\quad \text{in}\ \overline{\Om\setminus B_*}.
\end{aligned}
\end{equation}
Hence, by (\ref{3.33}), the function $\tilde{u}:=e^{\frac{V}{2\e^2}}u$ satisfies the estimate
\begin{equation}\label{3.34}
\|e^{-V}\tilde{u}\|_{C(\overline{B_*})} \leqslant C
\| \tilde{u}\|_{C(\overline{B_*})}\leqslant C
\e^{-n-3} \left(\|f\|_{\H^{\lf\frac{n}{2}\rf-1}(B_{\rho_1})} + \|f\|_{L_2(\Om)}
 + \|u\|_{L_2(\Om)}\right).
\end{equation}

Due to (\ref{3.4}), the functions $\tilde{u}$ and $e^{-V}\tilde{u}$ solve the boundary value problems
\begin{align}
&(-\e^2\D+\nabla V\cdot \nabla -\l)\tilde{u}=\tilde{f}:=e^{\frac{V}{2\e^2}}f\quad\text{in}\quad\Om,\qquad \tilde{u}=0\quad\text{on}\quad\p\Om,\nonumber
\\
&
\begin{aligned}
& \Big(-\e^2\D+(1-2\e^2)\nabla V\cdot\nabla+(1-\e^2)|\nabla V|^2-\e^2\D V-\l \Big)e^{-V}\tilde{u}=e^{-V}\tilde{f} \quad\text{in}\quad\Om,
\quad 
e^{-V} \tilde{u}=0\quad\text{on}\quad\p\Om.
\end{aligned}
\label{3.37}
\end{align}
Thanks to the second and third estimates in (\ref{3.36}), the bound $\l\leqslant \e^2$ and Lemma~\ref{lm3.1}, the potential in the latter problem has a lower bound:
\begin{equation*}
(1-\e^2)|\nabla V|^2-\e^2\D V-\l \geqslant C\e^{2-\frac{4}{k}}\quad\text{in}\quad \overline{\Om\setminus B_*}
\end{equation*}
and we can apply the classical maximum principle to (\ref{3.37}). Together with (\ref{3.34}) this yields:
\begin{equation*}
\|\tilde{u}\|_{C(\overline{\Om})}\leqslant C\|e^{-V}\tilde{u}\|_{C(\overline{\Om})} \leqslant C\e^{-n-5+\frac{4}{k}} \left(\|f\|_{\H^{\lf\frac{n}{2}\rf-1}(B_{\rho_1})} + \|\tilde{f}\|_{C(\overline{\Om})}
 + \|u\|_{L_2(\Om)}\right),
\end{equation*}
which implies (\ref{3.23b}). In the same way how (\ref{3.26a}), (\ref{3.26b}) were obtained, by the above estimate and the identities
\begin{gather*}
\e^2\|\nabla\tilde{u}\|_{L_2(\Om)}^2 =(\tilde{f},\tilde{u})_{L_2(\Om)} + \l\|\tilde{u}\|_{L_2(\Om)}^2 -(\nabla V\cdot\nabla\tilde{u},\tilde{u})_{L_2(\Om)}=(\tilde{f},\tilde{u})_{L_2(\Om)} +\Big((\tfrac{1}{2}\D V+\l)\tilde{u},\tilde{u}\Big)_{L_2(\Om)},
\\
-\D\tilde{u}=\e^{-2}(\tilde{f}-\nabla V\cdot\nabla\tilde{u}+\l\tilde{u}),
\end{gather*}
we are led  to (\ref{3.29}), (\ref{3.30}).
\end{proof}

\section{Formal asymptotics for eigenfunction}

In this section we make the first step in the proofs of Theorems~\ref{thEV},~\ref{thEF}, namely, we construct a formal asymptotic expansion for the eigenfunction $\Psi_\e$  associated with the lowest eigenvalue $\l_\e$ of the operator $\Op_\e$; the function $\Psi_\e$ is not supposed to be normalized in $L_2(\Om)$.

We rewrite the eigenvalue equation $\Op_\e\Psi_\e=\l_\e\Psi_\e$
as the boundary value problem
\begin{equation}\label{3.15}
(-\e^2\D + \nabla V\cdot \nabla)\Psi_\e=\l_\e\Psi_\e\quad\text{in}\quad \Om,\qquad \Psi_\e=0\quad\text{on}\quad\p\Om,
\end{equation}
and we are going to construct an asymptotic solution to this problem.  We shall construct a power in $\e$ asymptotic solution and since
 $\l_\e$ is exponentially small by Lemma~\ref{lm3.1},
 in all our further construction we shall in fact neglect the right hand side in the equation in (\ref{3.15}).

The constant function $\mathds{1}$  solves the homogeneous equation $(-\e^2\D +\nabla V\cdot \nabla)\mathds{1}=0$ in $\Om$  but does not satisfy the homogeneous Dirichlet  condition on $\p\Om$. In view of this, we choose $\mathds{1}$ as the leading term in the asymptotics for $\Psi_\e$ by adopting the following ansatz:
 \begin{equation}\label{3.16}
\Psi_\e(x)=1 -Q^\e(\z,s),\qquad Q^\e(\z,s)=\sum\limits_{j=0}^{\infty}\e^{2j}Q_j(\z,s),\qquad \z:=\e^{-2}\tau,
\end{equation}
where  $Q^\e$ is a boundary layer introduced by the asymptotic series and   $Q_j$ are some functions to be determined.

The boundary condition in (\ref{3.15})   imply immediately those for $Q_j$:
\begin{equation}
 Q_0(0,s)=1,\qquad Q_j(0,s)=0,\qquad j\geqslant 2.\label{4.22}
\end{equation}
Since  $Q^\e$ is a boundary layer,   the functions   $Q_j$ should decay exponentially at infinity, namely,
\begin{equation}\label{4.21}
 Q_j(\z,s)=O(e^{-\g_j\z}),\qquad \z\to+\infty,
\end{equation}
uniformly in $s$ and $\g_j$ are some positive numbers.

The Laplace operator and the gradient are rewritten in terms  of the variables $(s,\tau)$ as follows:
\begin{align}\label{4.14}
&\D=\frac{\p^2\ }{\p\tau^2}  +\frac{\p\ln J}{\p\tau}\frac{\p\ }{\p\tau}
+\dvr_s  L \nabla_s+\nabla_s\ln J\cdot L\nabla_s,
\qquad
\nabla V\cdot\nabla = \frac{\p V}{\p\tau}\frac{\p\ }{\p\tau} + L\nabla_s\cdot\nabla_s,
\\
&\nonumber
L=L(\tau,s):=(I-\tau b^t)^{-1}g^{-1}(I-\tau b)^{-1},\qquad J=J(\tau,s):=\sqrt{\det g}\det(I-\tau b),
\end{align}
where $I$ stands for the unit $(n-1)\times(n-1)$ matrix,    and, we recall, $g$ is the metric tensor  on $\p\Om$ and $b_{ij}=b_{ij}(s)$ is the second fundamental form on the inward side of $\p\Om$.

As $\tau\to+0$, by the smoothness of $V$ in the vicinity of $\p\Om$ (see (\ref{2.1a})) and the smoothness of $L$ and $J$ we have the Taylor series:
\begin{equation}
\begin{aligned}
&\ln J(\tau,s)=\sum\limits_{j=0}^{\infty} \tau^j \Tht_j(s),\qquad
L(\tau,s)=\sum\limits_{j=0}^{\infty}\tau^j L_j(s),\qquad
V(x)=\sum\limits_{j=0}^{\infty} \tau^j\tht_j(s),
\qquad
\tht_j(s):=\frac{1}{j!}\frac{\p^j V}{\p\tau^j}\bigg|_{\p\Om},
\end{aligned}\label{3.17}
\end{equation}
where $\tht_j,\Tht_j\in C^\infty(\p\Om)$ are functions and $L_j\in C^\infty(\p\Om)$ are some matrices. Since the matrix $L$ is obviously Hermitian, the same holds for $L_j$.
We substitute the above formulae and  (\ref{3.16}),   (\ref{4.14}) into the equation in (\ref{3.15}), neglect the right hand side, pass to the variable $\z$, and equate the coefficients at the like powers of $\e$. This gives the equation  for  $Q^\e$:
\begin{equation}\label{4.16}
 \cL^\e Q^\e=0\qquad\text{as}\qquad \z>0,
\end{equation}
where $\cL^\e$ denotes the differential expressions
\begin{equation}\label{4.17}
\begin{aligned}
\cL^\e:=&-\e^{-2}\frac{\p^2\ }{\p\z^2}+\e^{-2} \frac{\p V}{\p\tau}(\e^2\z,s)\frac{\p\ }{\p\z}-
\frac{\p\ln J}{\p\tau}(\e^2\z,s)\frac{\p\ }{\p\z} + L(\e^2\z,s)\nabla_s V\cdot\nabla_s
\\
&-\e^2 \dvr_s L(\e^2\z,s)\nabla_s -\e^2\nabla_s \ln J(\e^2\z,s)\cdot L(\e^2\z,s)\nabla_s.
\end{aligned}
\end{equation}

We substitute series for $Q^\e$ in (\ref{3.16}) and formulae (\ref{3.17}), (\ref{4.17}) into (\ref{4.16}) and expand the result into powers of  $\e$. This determines the equations for  $Q_j$:
\begin{gather}
\label{4.18b}
-\frac{\p^2 Q_j}{\p\z^2} + \tht_1 \frac{\p Q_j}{\p \z} =G_j,\qquad \z>0,\qquad G_j(\z,s):=\sum\limits_{i=0}^{j-1}\cL_i Q_{j-i-1},
\\
\begin{aligned}
\cL_i:=&\z^{i-1}\dvr_s L_{i-1}(s)\nabla_s -
\left((i+2)\z^{i+1} \tht_{i+2} - (i+1)\z^i\Tht_{i+1}\right)\frac{\p\ }{\p\z}
 - \sum\limits_{q=0}^{i} \z^  L_q\nabla_s (\tht_{i-q}-  \Tht_{i-q-1})\cdot\nabla_s.
\end{aligned}\nonumber
\end{gather}
Here we denote $L_j:=0$, $j\leqslant -1$, $\Tht_j:=0$, $j\leqslant -1$.

Equation (\ref{4.18b}) for $Q_0$ is homogeneous and since $\tht_1$ is negative and independent of $\z$, this equation possesses the only solution obeying boundary condition  (\ref{4.22}) and decay condition (\ref{4.21}):
\begin{equation}\label{4.23}
Q_0(\z,s)=e^{\tht_1(s)\z}.
\end{equation}
Calculating now the right hand side  $Q_1$, we see immediately that this is a product of the exponential $e^{\tht_1(s)\z}$ by a polynomial in $\z$ of first degree with infinitely differentiable in $s$ coefficients:
\begin{equation*}
G_1(\z,s)=  -\left(\Phi(s)\z  +
 \Tht_1(s)\tht_1(s)\right)e^{\tht_1(s)\z},
\end{equation*}
where $\Phi$ was defined in (\ref{2.12}).
Solving then problem   (\ref{4.18b}), (\ref{4.22}), (\ref{4.21}) for $Q_1$, we see that the solutions is of the same structure:
$Q_1(\z,s)=\Phi_1(\z,s)e^{\tht_1(s)\z}$,
where $\Phi_1$ is from (\ref{2.12}). Proceeding in the same way with other problems, by induction we prove easily  the following lemma.

\begin{lemma}\label{lm4.2}
Problems   (\ref{4.18b}),  (\ref{4.22}), (\ref{4.21}) are uniquely solvable and their solutions are of the form
$Q_j(\z,s)=\Phi_j(\z,s)e^{\tht_1(s)\z}$,
where  $\Phi_j$ are some polynomials of degree at most  $2j$ with infinitely differentiable in $s$ coefficients.
\end{lemma}

We choose an integer $N$ large enough and in view of the above construction, the proposed approximation for the eigenfunction reads as follows:
\begin{equation}
\Psi_{\e,N}(x):=1-Q^\e_N(x),\qquad Q^\e_N(x):=      \chi(x)\sum\limits_{j=0}^{N}\e^{2j} Q_j(\e^{-2}\tau,s). \label{4.28}
\end{equation}

Denote
$\mu(\e):=\int\limits_{\p\Om}e^{-\frac{\tht(s)}{\e^2}}\di s$.
The function $\mu(\e)$ is obviously positive. By the Laplace method, this function behaves as
\begin{equation}\label{5.0}
\mu(\e)=c_8\e^{\frac{2(n-1)}{p}} e^{-\frac{\tht_{min}}{\e^2}}\Big(1+O\big(\e^{\frac{2}{p}}\big)\Big),
\end{equation}
where $c_8$ is a positive constant independent of $\e$, $p\in\mathds{N}$ is some fixed number and it is the same as in (\ref{2.20}).

The next lemma states that $\psi_{\e,N}$ is a formal asymptotic solution of  problem (\ref{3.15}) and this completes the formal construction.

\begin{lemma}\label{lm6.1}
The function  $\Psi_{\e,N}$ is infinitely differentiable in $\overline{\Om}$ and solves the equation
\begin{equation*}
 \Op_\e \Psi_{\e,N}=h_{\e,N},
\end{equation*}
where $h_{\e,N}\in C^\infty(\overline{\Om})$  is a function supported in $\{x\in\Om:\, 0\leqslant \tau\leqslant \tau_0\}$ and obeying the estimates
\begin{equation*}
\|e^{-\frac{V}{2\e^2}}h_{\e,N}\|_{L_2(\Om)}=O\big(\e^{2N+1}
\mu^\frac{1}{2}(\e)\big),\qquad \|e^{-\frac{V}{2\e^2}}h_{\e,N}\|_{C(\overline{\Om})}= O\big( \e^{2N}e^{-\frac{\tht_{min}}{2\e^2}}\big),
 \qquad
  \big\|h_{\e,N}\big\|_{C(\overline{\Om})} =O(\e^{2N}).
\end{equation*}
\end{lemma}

\begin{proof}
The stated smoothness of the function   $\psi_{\e,N}$ is obvious. It vanishes on $\p\Om$ thanks to boundary conditions (\ref{4.22}). Hence, $\psi_{\e,N}\in\Dom(\Op^\e)$.
By straightforward calculations we obtain:
\begin{align*}
\big(\Op_\e\psi_{\e,N}\big)(x)=&   \big(\e^2\D-\nabla V\cdot \nabla\big) \chi(x)\sum\limits_{j=0}^{N}\e^{2j} Q_j(\e^{-2}\tau,s)
\\
=&  \chi(x) \cL^\e\sum\limits_{j=0}^{N}\e^{2j} Q_j(\z,s)
+    \sum\limits_{j=0}^{N}\e^{2j}
\Big(2\e^2\nabla\chi \cdot \nabla  Q_j(\e^{-2}\tau,s)
\\
&
+\e^2 Q_j(\e^{-2}\tau,s)\D\chi-Q_j(\e^{-2}\tau,s)\nabla V\cdot\nabla \chi(x)\Big)
=:h_{\e,N}^{(1)}+h_{\e,N}^{(2)}=:h_{\e,N}.
\end{align*}
The functions $h_{\e,N}^{(1)}$, $h_{\e,N}^{(2)}$, $h_{\e,N}$ are obviously infinitely differentiable in $\overline{\Om}$. Employing equations  (\ref{4.18b}) and Lemma~\ref{lm4.2}, we get:
\begin{equation*}
\big\|h_{\e,N}^{(1)} \big\|_{C(\overline{\Om})}\leqslant C\e^{2N},\qquad \big\| h_{\e,N}^{(2)}\big\|_{C(\overline{\Om})}\leqslant
 C\e^{-2N}e^{-\frac{\g c_2}{3\e^2}}.
\end{equation*}
Since $V=\tht_0+\tau\tht_1+O(\tau^2)$ for small $\tau$, by Lemma~\ref{lm4.2} we also obtain:
\begin{align*}
&\big\|e^{-\frac{V}{2\e^2}}h_{\e,N}^{(1)} \big\|_{L_2(\Om)} \leqslant C
\e^{2N+1}\mu^\frac{1}{2}(\e), && \big\|e^{-\frac{V}{2\e^2}}h_{\e,N}^{(2)}\big\|_{L_2(\Om)}\leqslant
 C\e^{1-2N}e^{-\frac{c_2\d}{6\e^2}}\mu^\frac{1}{2}(\e),
 \\
& \big\|e^{-\frac{V}{2\e^2}}h_{\e,N}^{(1)} \big\|_{C(\overline{\Om})} \leqslant C
\e^{2N}e^{-\frac{\tht_{min}}{2\e^2}}, && \big\|e^{-\frac{V}{2\e^2}}h_{\e,N}^{(2)}\big\|_{C(\overline{\Om})}\leqslant
 C\e^{-2N}e^{-\frac{c_2\d}{6\e^2}}e^{-\frac{\tht_{min}}{2\e^2}}.
\end{align*}
The proof is complete.
\end{proof}

\section{Asymptotics for the lowest eigenvalue and associated eigenfunction}

In this section we complete the proofs of  Theorems~\ref{thEV},~\ref{thEF}.
Throughout this section, the symbol $C$ stands for various inessential constants independent of $\e$, $x$, $s$, $\tau$ but possibly depending on $N$. We recall that the operator $\Op^\e$  is self-adjoint.

Let $\psi^\e$ be the eigenfunction of the operator $\Op^\e$ associated with the lowest eigenvalue $\l_\e$ and normalized in $L_2(\Om)$. By the standard smoothness improving theorems and (\ref{2.1a}), we infer that $\psi^\e\in C^2(\overline{\Om})\cap C^\infty(\overline{B_{\rho_1}})$. Then we apply Lemma~\ref{lm:Max} with $\l=\l_\e$, $u=\psi^\e$ and thanks to the normalization of $\psi^\e$, we get the estimates
\begin{equation}\label{6.22}
\begin{aligned}
&\|\psi^\e\|_{C(\overline{\Om})}\leqslant C\e^{-n-2}, &&
\|\nabla\psi^\e\|_{L_2(\Om)}\leqslant C\e^{-2}, && \|\p^2_{xx}\psi^\e\|_{L_2(\Om)}\leqslant C\e^{-4},
\\
&\|e^{\frac{V}{2\e^2}}\psi^\e\|_{C(\overline{\Om})}\leqslant C\e^{-n-5}, && \|\nabla e^{\frac{V}{2\e^2}}\psi^\e\|_{C(\overline{\Om})}\leqslant C\e^{-n-6},
\\
&\|\p^2_{xx} e^{\frac{V}{2\e^2}}\psi^\e\|_{C(\overline{\Om})}\leqslant C\e^{-n-8}, \qquad  && \|\psi^\e\|_{\H^{m+1}(B_{\rho_1})}\leqslant C\e^{-2m-2}, \qquad && m\in\mathds{N}.
\end{aligned}
\end{equation}
The latter estimates first are obtained for the norms on $B_{\frac{\rho_1}{2}}$. But then we can just lessen $\rho_1$ twice to get the stated inequalities.

By $L_2^\bot(\Om)$ we denote the orthogonal complement to $\psi^\e$ in $L_2(\Om)$. It clear that the spectrum of the restriction of the operator $\Op^\e$ on $L_2^\bot(\Om)\cap\Dom(\Op^\e)$ starts at $\l_2^\e$ and hence, by Lemma~\ref{lm3.1},
\begin{equation}\label{3.0}
\|(\Op^\e-\l)^{-1}\|\leqslant C\e^{\frac{4}{k}-2}\qquad \text{for all}\quad \l\in[0,\e^2].
\end{equation}

Denote
\begin{equation*}
\Psi_\e^N:=\cE_\e^{-1}\psi_{\e,N}, \qquad \Psi_N^\e(x)=e^{-\frac{V(x)}{2\e^2}}-\chi(x)e^{-\frac{V(x)}{2\e^2}} \sum\limits_{j=0}^{N}\e^{2j} Q_j(\e^{-2}\tau,s).
\end{equation*}
It follows from (\ref{3.4}) and Lemma~\ref{lm6.1} that
\begin{equation}\label{6.0}
\begin{aligned}
&
\Op^\e\Psi_N^\e=h_N^\e,\qquad && h_N^\e:=\cE_\e^{-1}h_{\e,N}=e^{-\frac{V}{2\e^2}}h_{\e,N}, \qquad &&\Psi_N^\e, h_N^\e\in C^\infty(\overline{\Om}),
\\
& \|h_N^\e\|_{L_2(\Om)}=O(\e^{2N+1}\mu^\frac{1}{2}(\e)), &&
 \|h_N^\e\|_{C(\overline{\Om})}=O\big(\e^{2N}e^{-\frac{\tht_{min}}{2\e^2}}\big),
 && \|e^{\frac{V}{2\e^2}}h_N^\e\|_{C(\overline{\Om})}=O(\e^{2N}),
\end{aligned}
\end{equation}
and $\supp h_N^\e\subset\big\{x:\, 0\leqslant \tau\leqslant \tau_0\big\}$.
The function $\Psi_N^{\e,\bot}:=\Psi_N^\e-(\Psi_N^\e,\psi^\e)_{L_2(\Om)}\psi^\e$ belongs to $L_2^\bot(\Om)\cap\Dom(\Op^\e)$ and  solves the equation
\begin{equation*}
\Op^\e\Psi_N^{\e,\bot}=h_N^{\e,\bot},\qquad h_N^{\e,\bot}:=h_N^\e-(h_N^\e,\psi^\e)_{L_2(\Om)}\psi^\e,
\end{equation*}
where the right hand side satisfies the estimates
\begin{equation}\label{6.24}
\|h_N^{\e,\bot}\|_{L_2(\Om)}\leqslant \|h_N^\e\|_{L_2(\Om)}\leqslant C\e^{2N+1}\mu^\frac{1}{2}(\e),\qquad \big|(h_N^\e,\psi^\e)_{L_2(\Om)}\big|=O\big(\e^{2N+1}\mu^\frac{1}{2}(\e)\big).
\end{equation}
It follows from  (\ref{3.0}) that
\begin{equation*}
\|\Psi_N^{\e,\bot}\|_{L_2(\Om)}\leqslant C\e^{2N+\frac{4}{k}-1}\mu^{\frac{1}{2}}(\e)\leqslant C\e^{2N-1}\mu^{\frac{1}{2}}(\e).
\end{equation*}

By  the normalization of $\psi^\e$ and Lemma~\ref{lm4.2}, the identity
\begin{equation*}
(\Psi_{N+1}^\e,\psi^\e)_{L_2(\Om)}-(\Psi_N^\e,\psi^\e)_{L_2(\Om)}=-\e^{2N+2} \int\limits_{\Om} e^{-\frac{V(x)}{2\e^2}} \chi(x) Q_{N+1}(\tau\e^{-2},s)\psi^\e(x)\di x=O\big(\e^{2N+3}\mu^\frac{1}{2}(\e)\big)
\end{equation*}
holds. Hence, there exists a function $A(\e)$ such that
\begin{equation}\label{5.11}
A(\e)=(\Psi_N^\e,\psi^\e)_{L_2(\Om)} + O\big(\e^{2N+1}\mu^\frac{1}{2}(\e)\big),\qquad |A(\e)|\leqslant C \|\psi_\infty^\e\|_{L_2(\Om)}\leqslant C.
\end{equation}
We let
 $\Psi^\e:=A(\e)\psi^\e$, $\Xi_N^\e:=\Psi_N^\e-\Psi^\e$,  and we see that
\begin{equation}\label{6.26}
\Xi_N^\e=\Psi_N^{\e,\bot}-\big(A(\e)-(\Psi_N^\e,\psi^\e)_{L_2(\Om)}\big)\psi^\e.
\end{equation}
Applying  Lemma~\ref{lm:Max} with $u=\Psi_N^{\e,\bot}$, $\l=0$, and taking into consideration estimates (\ref{6.22}), (\ref{6.0}), (\ref{6.24}),  we obtain a series of estimates for $\Psi_N^{\e,\bot}$. Using then (\ref{5.11}) and (\ref{6.22}), we can estimate the second term in the right hand in (\ref{6.26}). As a result, we obtain the following inequalities for $\Xi_N^\e$:
\begin{align}\label{6.31}
&
\begin{aligned}
&\|\Xi_N^\e\|_{L_2(\Om)}=O\big(\e^{2N-1}\mu^\frac{1}{2}(\e)\big), &&
\|\nabla \Xi_N^\e\|_{L_2(\Om)}=O\big(\e^{2N-3}\mu^\frac{1}{2}(\e)\big),
\\
&
\|\p_{xx}^2\Xi_N^\e\|_{L_2(\Om)}= O\big(\e^{2N-5}\mu^\frac{1}{2}(\e)\big),\hphantom{\|\Xi_N^\e\|_{C\overline{\Om}}} &&
\|\Xi_N^\e\|_{C(\overline{\Om})}= O\big(\e^{2N-2n}
e^{-\frac{\tht_{min}}{2\e^2}}
\big),
\end{aligned}
\\
&\label{6.32}
\begin{aligned}
&\|e^{\frac{V}{2\e^2}}\Xi_N^\e\|_{L_2(\Om)}= O\big(\e^{2N-3n-2} \big), &&
\|\nabla e^{\frac{V}{2\e^2}} \Xi_N^\e\|_{L_2(\Om)}=O\big(\e^{2N-3n-3} \big),
\\
&
\|\p_{xx}^2 e^{\frac{V}{2\e^2}}\Xi_N^\e\|_{L_2(\Om)}= O\big(\e^{2N-3n-5}\big),\qquad \hphantom{...}  &&
\|e^{\frac{V}{2\e^2}}\Xi_N^\e\|_{C(\overline{\Om})}= O\big(\e^{2N-3n-2} \mu^\frac{1}{2}(\e)\big),
\end{aligned}
\\
&
\begin{aligned}
\|\Xi_N^\e\|_{\H^{m+1}(B_{\frac{\rho_1}{2}})}= O\big(\e^{2N-4m+1}\mu^\frac{1}{2}(\e)\big),\qquad m\in\mathds{N}.
\end{aligned}\label{6.33}
\end{align}
The first estimate in (\ref{6.31}) yields that
\begin{equation*}
A(\e)=\|\Psi^\e\|_{L_2(\Om)}=\|\Psi_N^\e\|_{L_2(\Om)}+ O\big(\e^{2N+\frac{2n}{k}-1}\mu^{\frac{1}{2}}(\e)\big) =\|\psi_\infty^\e\|_{L_2(\Om)}(1+o(1))\ne0
\end{equation*}
and hence, $\Psi^\e$ is an eigenfunction of $\Op^\e$ associated with $\l_\e$.

It is easy to confirm that
\begin{align*}
& \big\|\e^{2j} \chi Q_j\big\|_{L_2(\Om)}=O\big(\e^{2j+1} \big), &&
  \big\|\e^{2j}\nabla  \chi Q_j\big\|_{L_2(\Om)}=O\big(\e^{2j-1} \big), \\
& \big\|\e^{2j}\p^2_{xx} \chi Q_j\big\|_{L_2(\Om)}=O\big(\e^{2j-3} \big),
&& \big\|\e^{2j} \chi Q_j\big\|_{C(\overline{\Om})} =O\big(\e^{2j} \big),
\\
& \big\|\e^{2j}e^{-\frac{V}{2\e^2}}\chi Q_j\big\|_{L_2(\Om)}=O\big(\e^{2j+1}\mu^{\frac{1}{2}}(\e)\big), &&
  \big\|\e^{2j}\nabla e^{-\frac{V}{2\e^2}}\chi Q_j\big\|_{L_2(\Om)}=O\big(\e^{2j-1}\mu^{\frac{1}{2}}(\e)\big), \\
& \big\|\e^{2j}\p^2_{xx}e^{-\frac{V)}{2\e^2}}\chi Q_j\big\|_{L_2(\Om)}=O\big(\e^{2j-3}\mu^{\frac{1}{2}}(\e)\big),
&& \big\|\e^{2j}e^{-\frac{V}{2\e^2}}\chi Q_j\big\|_{C(\overline{\Om})} =O\big(\e^{2j}e^{-\frac{\tht_{min}}{2\e^2}}\big).
\end{align*}
By the first of the above estimates and the first estimate in (\ref{6.31}) we infer that
\begin{equation}\label{6.12a}
\|  \Xi_{N-2}^\e\|_{L_2(\Om)}\leqslant \| \Xi_N^\e\|_{L_2(\Om)} + \|\e^{2(N-1)}\chi(Q_{N-1}+\e^2 Q_{N-2}\|_{L_2(\Om)}= O \big(\e^{2N-1}\mu^\frac{1}{2}(\e)\big),
\end{equation}
and since $N$ is arbitrary, $\|\Xi_N^\e\|_{L_2(\Om)}=\big(\e^{2N+3}\mu^\frac{1}{2}(\e)\big)$. In the same way we can improve other estimates in (\ref{6.31}), (\ref{6.32}):
\begin{align}\label{6.27}
&
\begin{aligned}
&\|\Xi_N^\e\|_{L_2(\Om)}=O\big(\e^{2N+3}\mu^\frac{1}{2}(\e)\big), &&
\|\nabla \Xi_N^\e\|_{L_2(\Om)}=O\big(\e^{2N+1}\mu^\frac{1}{2}(\e)\big),
\\
&
\|\p_{xx}^2\Xi_N^\e\|_{L_2(\Om)}= O\big(\e^{2N-1}\mu^\frac{1}{2}(\e)\big),\hphantom{\|\Xi_N^\e\|_{C\overline{\Om}}} &&
\|\Xi_N^\e\|_{C(\overline{\Om})}= O\big(\e^{2N+2}e^{-\frac{\tht_{min}}{2\e^2}}\big),
\end{aligned}
\\
&\label{6.29}
\begin{aligned}
&\|e^{\frac{V}{2\e^2}}\Xi_N^\e\|_{L_2(\Om)}= O\big(\e^{2N+3}\big), &&
\|\nabla e^{\frac{V}{2\e^2}} \Xi_N^\e\|_{L_2(\Om)}=O\big(\e^{2N+1}\big),
\\
&
\|\p_{xx}^2 e^{\frac{V}{2\e^2}}\Xi_N^\e\|_{L_2(\Om)}= O\big(\e^{2N-1}\big),\hphantom{\|\nabla e^{\frac{V}{2\e}} \Xi^\e\|}  &&
\|e^{\frac{V}{2\e^2}}\Xi_N^\e\|_{C(\overline{\Om})}= O\big(\e^{2N+2}\big).
\end{aligned}
\end{align}
It follows from the definition of $\Psi_\e^N$, $\Psi_{\e,N}$, $\Xi_N^\e$, $\Psi^\e$ and (\ref{5.11}) that the eigenfunction $\Psi_\e:=A(\e)\psi^\e$  of the operator $\Op_\e$ satisfies (\ref{2.10}) with $\Xi_{\e,N}:=e^{\frac{V}{2\e^2}}\Xi_N^\e$ and by (\ref{6.29}), the error term satisfies (\ref{2.15b}), (\ref{2.15c}).

Let us prove estimates (\ref{2.14a}). For each $\om\subset\Om$, by first two estimates in (\ref{6.31}) and (\ref{5.0}), we have  the relations
\begin{equation}\label{6.21}
\begin{aligned}
&\|\Xi_N^\e\|_{L_2(\om)}=\|e^{\frac{V}{2\e^2}} e^{-\frac{V}{2\e^2}}\Xi_N^\e\|_{L_2(\om)}\leqslant
e^{\frac{\V_\om}{2\e^2}} \|e^{-\frac{V}{2\e^2}}\Xi_N^\e\|_{L_2(\Om)} =O\Big(\e^{2N+3} e^{-\frac{\tht_{min}-\V_\om}{2\e^2}}\Big),
\\
&\|\nabla \Xi_N^\e\|_{L_2(\om)}=\|\nabla e^{\frac{V}{2\e^2}} e^{-\frac{V}{2\e^2}}\Xi_N^\e\|_{L_2(\om)}\leqslant
\|e^{\frac{V}{2\e^2}} \nabla e^{-\frac{V}{2\e^2}}\Xi_N^\e\|_{L_2(\om)} + C\e^{-2}\|\Xi_N^\e\|_{L_2(\om)}
\\
&\hphantom{\|\nabla \Xi_N^\e\|_{L_2(\om)}} =O\Big(\e^{2N+1} e^{-\frac{\tht_{min}-\V_\om}{2\e^2}}\Big).
\end{aligned}
\end{equation}
Other two estimates in (\ref{2.14a}) can be proved in the same way.

Let us find the asymptotics for $\l_\e$.
The function $\Psi^\e$ solves the boundary value problem
\begin{equation*}
(-\e^2\D+W_\e)\Psi^\e=\l_\e \Psi^\e\quad\text{in}\quad\Om, \qquad \Psi^\e=0\quad\text{on}\quad\p\Om.
\end{equation*}
We multiply the  equation by $\psi_\infty^\e$ and integrate twice by parts over $\Om$. This gives  identity:
\begin{equation}\label{5.14}
\l_\e=\frac{\e^2\int\limits_{\p\Om} \psi_\infty^\e \frac{\p\Psi^\e}{\p\tau}\di s} {(\Psi^\e,\psi_\infty^\e)_{L_2(\Om)}},
\end{equation}
which coincides with (\ref{5.14a}) thanks to the definition of $\Psi^\e$ and $\psi_\e^\infty$. By the Laplace method,  in view of (\ref{2.1a}), (\ref{2.6a}),  we get the asymptotic expansion
\begin{equation}\label{5.17}
\|\psi_\infty^\e\|_{L_2(\Om)}^2= \sum\limits_{j=n}^{\infty}\a_j\e^{\frac{2j}{k}},
\end{equation}
where $\a_j$ are some real  numbers and, in particular,
\begin{equation}\label{5.18}
\a_n=\int\limits_{\mathds{R}^n} e^{-V_0}\di x,\qquad \a_{n+1}=\int\limits_{\mathds{R}^n} V_1 e^{-V_0}\di x,
\end{equation}
where, we recall, the functions $V_0$, $V_1$ were defined in (\ref{5.28a}).
In view of definition (\ref{4.28}) of the function $\Psi^\e_N$, we also have
\begin{equation*}
\|\Psi^\e_N-\psi_\infty^\e\|_{L_2(\Om)}=O\big(\e\mu^\frac{1}{2}(\e)\big).
\end{equation*}
The above estimate, (\ref{6.27}) and the standard embedding theorems
yield
\begin{equation}\label{5.20}
(\Psi^\e,\psi_\infty^\e)_{L_2(\Om)}= \sum\limits_{j=n}^{\infty}\a_j\e^{\frac{2j}{k}},
\qquad \left\|\frac{\p\Psi^\e}{\p\tau}-\frac{\p\Psi^\e_N}{\p\tau}\right\|_{L_2(\p\Om)} =O\big(\e^{2N -1}\mu^{\frac{1}{2}}(\e)\big).
\end{equation}
We also observe that
$\|\psi_\infty^\e\|_{L_2(\p\Om)}^2=\mu(\e)$.
Hence,
\begin{equation}\label{5.22}
\int\limits_{\p\Om} \psi_\infty^\e\frac{\p\Psi^\e}{\p\tau}\di s=\int\limits_{\p\Om}\psi_\infty^\e\frac{\p\Psi^\e_N}{\p\tau}\di s+O\big(\e^{2N-1}\mu(\e)\big),
\end{equation}
Employing (\ref{4.28}),
(\ref{4.23}), we calculate:
\begin{align*}
\frac{\p\Psi^\e_N}{\p\tau}\bigg|_{\p\Om}= &   e^{-\frac{\tht_0(s)}{2\e^2}}\left(\frac{\tht_1(s)}{2\e^2} - \sum\limits_{j=1}^{N}\e^{2j-2} \frac{\p Q_j}{\p\z}(0,s)\right)
= e^{-\frac{\tht_0(s)}{2\e^2}}\left(\frac{\tht_1(s)}{2\e^2} - \sum\limits_{j=1}^{N}\e^{2j-2} \frac{\p \Phi_j}{\p\z}(0,s)\right).
\end{align*}
We substitute this identity into (\ref{5.22}) and for sufficiently large $N$ we arrive at the asymptotic formula
\begin{equation}\label{5.25}
\int\limits_{\p\Om} \psi_\infty^\e\frac{\p\Psi^\e}{\p\tau}\di s=\sum\limits_{j=0}^{N}\e^{2j-2}\mu_j(\e) +O\big(\e^{2N-1}\mu(\e)\big),
\end{equation}
where the functions $\mu_j$ are defined by (\ref{5.23}) and satisfy (\ref{5.24}).

Formulae (\ref{5.14}), (\ref{5.20}), (\ref{5.25}), (\ref{5.23}), (\ref{5.24}) gives the final asymptotics for $\l_\e$:
\begin{equation*}
\l_\e=\frac{\sum\limits_{j=0}^{N}\e^{2j-2}\mu_j(\e) +O\big(\e^{2N-1}\mu(\e)\big)}{\sum\limits_{j=n}^{\infty}\a_j\e^{\frac{2j}{k}}}
=\e^{-\frac{2n}{k}-2}\sum\limits_{j=0}^{ Nk}\e^{\frac{2j}{k}}M_j\left(\mu_0(\e),\ldots, \mu_{\lf\frac{j}{k}\rf}(\e)
\right) +O\big(\e^{\frac{2(N-n+1)}{k}-2}\mu(\e)\big),
\end{equation*}
where $M_j$ are some linear combinations with fixed coefficients and this proves (\ref{2.9}).
Formula (\ref{5.28}) is implied by
 (\ref{5.18}), (\ref{5.23}).
Due to the positivity of the functions $\tht_0$ and $-\tht_1$ we have
\begin{equation*}
C\mu_0 \leqslant M_0(\mu_0) \leqslant C^{-1}\mu_0,\qquad C>0.
\end{equation*}

\section{Asymptotics for solution}

In this section we prove Theorem~\ref{thEq}. Our strategy is as follows. First we construct the asymptotics for the solution of  equation (\ref{3.6}) recovering then the solution of (\ref{3.2}) by formula (\ref{3.5}). The solution of equation (\ref{3.6}) is given by $u^\e=(\Op^\e)^{-1}\psi_\infty^\e$ and we  represent it as
\begin{equation}\label{7.1}
u^\e=K_\e\Psi^\e+u^{\e,\bot},
\end{equation}
where $K_\e$ is some constant and $u^{\e,\bot}$ is orthogonal to $\Psi_\e$ in $L_2(\Om)$. The function $\psi_\infty^\e$ is explicit, while the asymptotics for $\Psi_\e$ was studied in details in the above given proof of Theorem~\ref{thEF}. And thanks to this asymptotics, it turns out that the right hand side in an equation for $u^{\e,\bot}$ has a structure of boundary layer like $Q^\e$ in (\ref{3.16}). This implies the same structure for $u^{\e,\bot}$. By a simple trick we also succeed to describe the asymptotics for $K_\e$ and this gives a complete asymptotic expansion for $u^\e$.

We proceed to the detailed proof.   We represent the solution of equation (\ref{3.6}) by formula (\ref{7.1}), where $u^{\e,\bot}\in \Dom(\Op^\e)\cap L_2^\bot(\Om)$ and $K_\e$ is some constant. Substituting this representation into (\ref{3.6}), we see immediately that $K_\e$ is given by
\begin{equation*}
  K_\e=\frac{(\psi_\infty^\e,\Psi^\e)_{L_2(\Om)}}{\l_\e \|\Psi^\e\|_{L_2(\Om)}^2},
\end{equation*}
which coincides with (\ref{7.2}) thanks to the identities
$\psi_\infty^\e=e^{-\frac{V}{2\e^2}}$, $\Psi^\e=e^{-\frac{V}{2\e^2}}\psi_\e$. The function $u^{\e,\bot}$ solves the equation
\begin{equation*}
\Op^\e u^{\e,\bot}=f^\e,\qquad f^\e:=\psi_\infty^\e- \frac{(\psi_\infty^\e,\Psi^\e)_{L_2(\Om)}}{\|\Psi^\e\|_{L_2(\Om)}^2} \Psi^\e=\psi_\infty^\e-\Psi^\e- \frac{(\psi_\infty^\e-\Psi^\e,\Psi^\e)_{L_2(\Om)}}{\|\Psi^\e\|_{L_2(\Om)}^2} \Psi^\e.
\end{equation*}
According (\ref{2.10}), for arbitrary $N\in\mathds{N}$, the function $\psi_\infty^\e -\Psi^\e$ can be represented as
$\psi_\infty^\e -\Psi^\e= e^{-\frac{V}{2\e^2}}Q_N^\e - e^{-\frac{V}{2\e^2}}\Xi_N^\e$.
Then we can rewrite $f^\e$ as
\begin{gather}
f^\e=Q_N^{\e,\bot}- \Xi_N^{\e,\bot},\nonumber
\\
\label{7.5b}
 Q_N^{\e,\bot}:=e^{-\frac{V}{2\e^2}}Q_N^\e - \frac{(e^{-\frac{V}{2\e^2}}Q_N^\e,\Psi^\e)_{L_2(\Om)}} {\|\Psi^\e\|_{L_2(\Om)}^2}\Psi^\e,\qquad
\Xi_N^{\e,\bot}:=e^{-\frac{V}{2\e^2}}\Xi_N^\e - \frac{(e^{-\frac{V}{2\e^2}}\Xi_N^\e,\Psi^\e)_{L_2(\Om)}} {\|\Psi^\e\|_{L_2(\Om)}^2}\Psi^\e.
\end{gather}
We  stress that $\Xi_N^{\e,\bot},\,\Xi_N^{\e,\bot}\in L_2^\bot(\Om)$
and by the first estimate in (\ref{6.27})
\begin{equation}\label{6.18}
\|\Xi_N^{\e,\bot}\|_{L_2(\Om)}=O\big(\e^{2N+3}\mu^\frac{1}{2}(\e)\big).
\end{equation}
The equations
\begin{equation}\label{7.7}
\Op^\e u_N^{\e,\bot}=Q_N^{\e,\bot},\qquad \Op^\e w_N^{\e,\bot}=\Xi_N^{\e,\bot}
\end{equation}
are uniquely solvable in $\Dom(\Op^\e)\cap L_2^\bot(\Om)$ and the solutions satisfy the identity
\begin{align}\label{7.7a}
u^{\e,\bot}=u_N^{\e,\bot} - w_N^{\e,\bot}.
\end{align}

Our next step is to construct an asymptotic expansion for $u_N^{\e,\bot}$. The main idea of the construction is as follows. First we construct an asymptotics of the solution to the equation
\begin{equation}\label{7.10}
\Op^\e u_N^\e=e^{-\frac{V}{2\e^2}} Q_N^\e.
\end{equation}
Once we do this, we see immediately that  the projection of $u_N^\e$ on $L_2^\bot(\Om)$ is exactly $u_N^{\e,\bot}$:
\begin{equation}\label{7.11}
u_N^{\e,\bot}=u_N^\e-\frac{(u_N^\e,\Psi^\e)_{L_2(\Om)}} {\|\Psi^\e\|_{L_2(\Om)}^2}\Psi^\e.
\end{equation}

Since the function $Q_N^\e$ is a boundary layer, we assume the same structure for $u_N^\e$:
\begin{equation}\label{7.13}
u_N^\e(x)=e^{-\frac{V}{2\e^2}} P^\e(x),\qquad P^\e(x):= \chi(x) \left(\sum\limits_{j=1}^{N+1} \e^{2j} P_j(\z,s)+O(\e^{2N+4})\right).
\end{equation}
We rewrite equation (\ref{7.10}) as the associated boundary value problem and substitute then identities (\ref{7.13}), (\ref{4.14}), (\ref{3.17}). Then we are led to the equation
$\cL^\e P^\e= Q^\e$ and  expanding it in powers of $\e$, we arrive at the boundary value problems for $P_j$ similar to (\ref{4.22}), (\ref{4.21}), (\ref{4.18b}):
\begin{equation}\label{7.14}
-\frac{\p^2 P_j}{\p\z^2} +  \tht_1\frac{\p P_j}{\p \z} =F_j,\quad \z>0,\qquad P_j(0,s)=0,\qquad
P_j(\z,s)=O(e^{-\g_j\z}),\quad \z\to+\infty,
\end{equation}
where $\g_j>0$ are some numbers and
$F_j:=Q_{j-1}+\sum_{i=0}^{j-2} \cL_i P_{j-i-1}$.

Due to (\ref{4.23}), the right hand side in the equation for $P_1$ reads as $F_1=Q_0$ and we  find $P_1(\z,s)=U_1(\z,s)e^{\tht_1(s)\z}$, where $U_1$ is given by formula (\ref{7.16}). Employing formula (\ref{2.12}) for $\Phi_2$, by straightforward calculations we also find  $P_2(\z,s)=U_2(\z,s)e^{\tht_1(s)\z}$, where $U_2$ is given by (\ref{7.16}).
Other functions $P_j$ can be found recurrently. The result is summarized in the following lemma, which can be easily proved by induction.

\begin{lemma}\label{lm7.1}
Problems   (\ref{7.14}) are uniquely solvable and their solutions are of the form $P_j(\z,s)=U_j(\z,s)e^{\tht_1(s)\z}$,
where  $U_j$ are some polynomials of degree $2j-1$ with infinitely differentiable in $s$ coefficients.
\end{lemma}

We denote
\begin{equation}\label{7.33}
P_N^\e(x):=\chi(x) \sum\limits_{j=1}^{N+1}\e^{2j} P_j(\tau\e^{-2},s).
\end{equation}

\begin{lemma}\label{lm7.3}
The function $P_N^\e$ is infinitely differentiable in $\overline{\Om}$, is supported in  $\{x\in\Om:\, 0\leqslant \tau\leqslant \tau_0\}$ and solves the equation
\begin{equation*}
\Op_\e  P_N^\e=  Q_N^\e+Z_N^\e.
\end{equation*}
Here $Z_N^\e\in C^\infty(\overline{\Om})$ is a function supported in $\{x\in\Om:\, 0\leqslant \tau\leqslant \tau_0\}$ and  obeying the estimate
\begin{equation*}
 \big\|Z_N^\e\big\|_{C(\overline{\Om})} =O(\e^{2N+2}), \qquad
\|e^{-\frac{V}{2\e^2}}Z_N^\e\|_{L_2(\Om)}=O\big(\e^{2N+3}
\mu^\frac{1}{2}(\e)\big),\qquad \|e^{-\frac{V}{2\e^2}}Z_N^\e\|_{C(\overline{\Om})}= O\big(\e^{2N+2}e^{-\frac{\tht_{min}}{2\e^2}}\big).
\end{equation*}
\end{lemma}

This lemma can be proved in the same way as Lemma~\ref{lm6.1}.

In view of formulae (\ref{7.11}), we let
\begin{equation}\label{7.21}
P_N^{\e,\bot}=e^{-\frac{V}{2\e^2}} P_N^\e-\frac{\big(e^{-\frac{V}{2\e^2}} P_N^\e,\Psi^\e\big)_{L_2(\Om)}} {\|\Psi^\e\|_{L_2(\Om)}^2}\Psi^\e, \qquad
P_N^{\e,\bot} \in \Dom(\Op^\e)\cap L_2^\bot(\Om).
\end{equation}
Lemma~\ref{lm7.3} and formulae (\ref{3.4}), (\ref{7.5b}) imply that this function solves the equation
\begin{equation}\label{7.22}
\Op^\e P_N^{\e,\bot}=Q_N^{\e,\bot} + Z_N^{\e,\bot},\qquad Z_N^{\e,\bot}:=e^{-\frac{V}{2\e^2}} Z_N^\e- \frac{\big(e^{-\frac{V}{2\e^2}} Z_N^\e,\Psi^\e\big)_{L_2(\Om)}}{\|\Psi^\e\|_{L_2(\Om)}^2}\Psi^\e,
\end{equation}
where the function $Z_N^{\e,\bot}$ is   infinitely differentiable in $\overline{\Om}$, belongs to $L_2^\bot(\Om)$ and satisfies the estimate
\begin{equation}\label{7.23}
\|Z_N^{\e,\bot}\|_{L_2(\Om)}=O\big(\e^{2N+1}\mu^{\frac{1}{2}}(\e)\big).
\end{equation}
By equations (\ref{7.22}), (\ref{7.7}) we infer that
\begin{equation}\label{7.25}
\Op^\e(u_N^{\e,\bot}-P_N^{\e,\bot})=-Z_N^{\e,\bot}.
\end{equation}

\begin{lemma}\label{lm:bot}
The estimates hold:
\begin{equation}\label{7.17}
\begin{gathered}
\frac{\big|\big(e^{-\frac{V}{2\e^2}}  P_N^\e,\Psi^\e\big)_{L_2(\Om)}\big|} {\|\Psi^\e\|_{L_2(\Om)}^2}= O\big(\e^{4-\frac{2n}{k}}e^{-\frac{c_6}{\e^2}}\mu^{\frac{1}{2}}(\e)\big),
\\
\frac{\big|\big(e^{-\frac{V}{2\e^2}}  Z_N^\e,\Psi^\e\big)_{L_2(\Om)}\big|} {\|\Psi^\e\|_{L_2(\Om)}^2}= O\big(\e^{2N+3-\frac{2n}{k}}e^{-\frac{c_9}{\e^2}}\mu^{\frac{1}{2}}(\e)\big),
\qquad
 \frac{\big|\big(e^{-\frac{V}{2\e^2}}  \Xi_N^\e,\Psi^\e\big)_{L_2(\Om)}\big|} {\|\Psi^\e\|_{L_2(\Om)}^2}= O\big(\e^{2N+3-\frac{n}{k}}\mu^{\frac{1}{2}}(\e)\big),
\end{gathered}
\end{equation}
where $0<c_9<\tfrac{\tht_{min}}{2}$ is some fixed constant independent of $\e$.
\end{lemma}

\begin{proof}
We recall that $\Pi_\d=\{x:\, 0<\tau< \d\}$. In view of (\ref{7.33}) and the asymptotics for $\Psi^\e$ established in the previous section we get:
\begin{equation*}
\big|\big(e^{-\frac{V}{2\e^2}} P_N^\e,\Psi^\e\big)_{L_2(\Om)}\big|=\|e^{-\frac{V}{2\e^2}} P_N^\e\|_{L_2(\Pi_\d)} \|\Psi^\e\|_{L_2(\Pi_\d)}
\leqslant  C\e^3\mu^\frac{1}{2}(\e) \|\psi_\infty^\e\|_{L_2(\Pi_\d)}
\leqslant
  C\e^4 e^{-\frac{c_9}{\e^2}}\mu^\frac{1}{2}(\e),\quad c_9:=\min\limits_{\overline{\Pi_\d}} V.
\end{equation*}
Employing now  (\ref{5.17}), (\ref{5.18}) and (\ref{2.10}), we arrive at the first formula in (\ref{7.17}). The second formula can be proved in the same way. The third formula is implied by (\ref{6.18}).
\end{proof}

We denote
\begin{equation}\label{6.37}
S_N^{\e,\bot}:=u_N^{\e,\bot}-P_N^{\e,\bot}-w_N^{\e,\bot}.
\end{equation}\
This function belongs to $L_2^\bot(\Om)\cap\Dom(\Op^\e)$ and it follows from (\ref{7.7}), (\ref{7.25}) that
$\Op^\e S_N^{\e,\bot}=\Xi_N^{\e,\bot}-Z_N^{\e,\bot}$.
Thanks to (\ref{7.23}), (\ref{6.18}) and (\ref{3.0}) we obtain:
\begin{equation}\label{6.34}
\|S_N^{\e,\bot}\|_{L_2(\Om)}=O\big(\e^{2N}\mu^\frac{1}{2}(\e)\big).
\end{equation}
By Lemma~\ref{lm7.3}, the function $Z_N^\e$ vanishes identically on $B_{\rho_1}$. Then Lemma~\ref{lm:bot} and estimates (\ref{6.33}), (\ref{6.27}), (\ref{6.29}) yield:
\begin{align*}
&\big\|e^{\frac{V}{2\e^2}}(\Xi_N^{\e,\bot} -Z_N^{\e,\bot})\big\|_{L_2(\Om)}=O(\e^{2N+2}),
&&
\big\|e^{\frac{V}{2\e^2}}(\Xi_N^{\e,\bot} -Z_N^{\e,\bot})\big\|_{C(\overline{\Om})}=O(\e^{2N+2}),
\\
&\big\| \Xi_N^{\e,\bot} -Z_N^{\e,\bot} \big\|_{C(\overline{\Om})}=O(\e^{2N+2}e^{-\frac{\tht_{min}}{2\e^2}}),
&&\big\| \Xi_N^{\e,\bot} -Z_N^{\e,\bot} \big\|_{L_2(\Om)}=O\big(\e^{2N+3}\mu^\frac{1}{2}(\e)\big),
\\
& \big\|\Xi_N^{\e,\bot} -Z_N^{\e,\bot}\big\|_{W_2^{m+1}(B_{\frac{\rho_1}{2}})}= O\big(\e^{2N-4m+1}\mu^\frac{1}{2}(\e)\big), && m\in\mathds{N}.
\end{align*}
We again replace $\rho_1$ by $\frac{\rho_1}{2}$ and thanks to the above estimates and (\ref{6.34}), we can apply Lemma~\ref{lm:Max} to $S_N^{\e,\bot}$  obtaining the estimates:
\begin{equation}\label{6.36}
\begin{aligned}
&\|\nabla S_N^{\e,\bot}\|_{L_2(\Om)}=O\big(\e^{2N-2}\mu^\frac{1}{2}(\e)\big), && \|\p_{xx}^2 S_N^{\e,\bot}\|_{L_2(\Om)}=O\big(\e^{2N-4}\mu^\frac{1}{2}(\e)\big),
\\
&\|S_N^{\e,\bot}\|_{C(\overline{\Om})}= O\big(\e^{2N-3n+3}e^{-\frac{\tht_{min}}{2\e^2}}\big),
&& \|e^{\frac{V}{2\e^2}} S_N^{\e,\bot}\|_{C(\overline{\Om})} = O(\e^{2N-3}),
\\
&\|\nabla e^{\frac{V}{2\e^2}} S_N^{\e,\bot}\|_{L_2(\Om)}=O\big(\e^{2N-3n-1}\big),\qquad  && \|\p_{xx}^2 e^{\frac{V}{2\e^2}} S_N^{\e,\bot}\|_{L_2(\Om)}=O\big(\e^{2N-3n-3}\big).
\end{aligned}
\end{equation}
We substitute (\ref{7.7a}), (\ref{6.37}), (\ref{7.21}) into (\ref{7.1}):
\begin{equation}\label{6.38}
u^\e=K_\e\Psi^\e +e^{-\frac{V}{2\e^2}} P_N^\e +e^{-\frac{V}{2\e^2}}
S_N^\e,\qquad S_N^\e:=e^{\frac{V}{2\e^2}} \left(S_N^{\e,\bot}
-\frac{\big(e^{-\frac{V}{2\e^2}} P_N^\e,\Psi^\e\big)_{L_2(\Om)}} {\|\Psi^\e\|_{L_2(\Om)}^2}\Psi^\e\right).
\end{equation}
Estimates (\ref{6.36}), (\ref{7.17}) and asymptotics for $\Psi^\e$ established in the previous section give a series of estimates for $S_{\e,N}$ similar to (\ref{6.36}). Then we can improve them as in (\ref{6.12a}) and also reproduce the calculations from (\ref{6.21}) for $S_{\e,N}$. This leads us to estimates (\ref{2.14}), (\ref{2.18a}), (\ref{2.18b}). Multiplying (\ref{6.38})  by $e^{\frac{V}{2\e^2}}$, we arrive at (\ref{2.13}).

It remains to find out the asymptotics for $K_\e$ to complete the proof of Theorem~\ref{thEq}. We first obtain one more formula for $K_\e$. Namely, we multiply equation (\ref{3.6}) by $\psi_\infty^\e$ and integrate then twice by parts over $\Om$. This gives:
\begin{equation*}
\|\psi_\infty^\e\|_{L_2(\Om)}^2=\e^2\int\limits_{\p\Om} \psi_\infty^\e\frac{\p u^\e}{\p\tau}\di s.
\end{equation*}
Then we substitute (\ref{6.38}) into the above identity:
\begin{equation*}
\|\psi_\infty^\e\|_{L_2(\Om)}^2=\e^2 K_\e\int\limits_{\p\Om} \psi_\infty^\e\frac{\p \Psi^\e}{\p\tau}\di s+\e^2
\int\limits_{\p\Om} \psi_\infty^\e\frac{\p\ }{\p\tau} e^{-\frac{V}{2\e^2}} P_N^\e\di s+\e^2
\int\limits_{\p\Om} \psi_\infty^\e\frac{\p\ }{\p\tau} e^{-\frac{V}{2\e^2}} S_N^\e\di s,
\end{equation*}
and hence, by the identity $\Psi^\e=e^{-\frac{V}{2\e^2}}\Psi_\e$
and the homogeneous Dirichlet condition for $P_N^\e$ and $S_N^\e$ on $\p\Om$,
\begin{equation*}
K_\e=\frac{\|\psi_\infty^\e\|_{L_2(\Om)}^2}{\e^2 \int\limits_{\p\Om} e^{-\frac{\tht}{\e^2}} \frac{\p \Psi_\e}{\p\tau}\di s}- \frac{\int\limits_{\p\Om}   e^{-\frac{\tht}{\e^2}}
\frac{\p P_N^\e}{\p\tau}\di s}{  \int\limits_{\p\Om} e^{-\frac{\tht}{\e^2}} \frac{\p \Psi_\e}{\p\tau}\di s}
- \frac{\int\limits_{\p\Om}  e^{-\frac{\tht}{\e^2}}\frac{\p S_N^\e}{\p\tau}\di s}{  \int\limits_{\p\Om} e^{-\frac{\tht}{\e^2}} \frac{\p \Psi_\e}{\p\tau}\di s}.
\end{equation*}
The first term in the right hand side
is exactly $K^{(exp)}_\e$, while
\begin{equation}\label{6.12}
K^{(pow)}_\e:= - \frac{\int\limits_{\p\Om} e^{-\frac{\tht}{\e^2}} \frac{\p P_N^\e}{\p\tau}\di s}{  \int\limits_{\p\Om} e^{-\frac{\tht}{\e^2}} \frac{\p \Psi_\e}{\p\tau}\di s}
- \frac{\int\limits_{\p\Om} e^{-\frac{\tht}{\e^2}} \frac{\p S_N^\e}{\p\tau}\di s}{ \int\limits_{\p\Om} e^{-\frac{\tht}{\e^2}} \frac{\p \Psi_\e}{\p\tau}\di s}.
\end{equation}
This proves representation (\ref{6.28}).

Let us find asymptotics for $K^{(exp)}_\e$ and $K^{(pow)}_\e$.
The asymptotics for the denominator in $K^{(exp)}_\e$   has already been found in (\ref{5.25}), while the asymptotics for $\|\psi_\infty^\e\|_{L_2(\Om)}^2$ is provided by (\ref{5.17}). Bearing relations (\ref{5.24}), (\ref{5.28}) in mind, we calculate easily the asymptotics for $K_\e^{(\exp)}$ as the quotient of series (\ref{5.25}) and (\ref{5.17}). This leads us to asymptotic series (\ref{6.10}).

By estimates (\ref{2.18a}) and the standard embedding theorems we infer that
\begin{equation*}
\int\limits_{\p\Om} e^{-\frac{\tht}{\e^2}}\frac{\p S_N^\e}{\p\tau}\di s=O\big(\e^{2N-1}\mu(\e)\big)
\end{equation*}
and thanks to (\ref{5.25}), (\ref{5.24}), (\ref{5.0}),
\begin{equation}\label{6.7}
\frac{\int\limits_{\p\Om} e^{-\frac{\tht}{\e^2}}\frac{\p S_N^\e}{\p\tau}\di s}{ \int\limits_{\p\Om} e^{-\frac{\tht}{\e^2}} \frac{\p \Psi_\e}{\p\tau}\di s} = O(\e^{2N-1}).
\end{equation}
Employing definition (\ref{7.33}) of the function $P_N^\e$ and Lemma~\ref{lm7.1}, we calculate the numerator in the first term in (\ref{6.12}):
\begin{equation}\label{6.8}
  \int\limits_{\p\Om} e^{-\frac{\tht}{\e^2}} \frac{\p P_N^\e}{\p\tau}\di s =
  \sum\limits_{j=0}^{N}\e^{2j} \int\limits_{\p\Om} e^{-\frac{\tht(s)}{\e^2}} \frac{\p U_{j+1}}{\p\z}(0,s)\di s=\sum\limits_{j=0}^{N}\e^{2j}\eta_{j+1}(\e),
\end{equation}
where $\eta_j$ were defined in (\ref{6.13}). By the Laplace method, we see that estimates (\ref{6.14}) hold  true. Calculating now the asymptotics for the first term in the right hand side in (\ref{6.12}) as the quotient of series (\ref{6.8}) and (\ref{5.25}), we obtain:
\begin{align*}
- \frac{\int\limits_{\p\Om} e^{-\frac{\tht}{\e^2}}\frac{\p P_N\e}{\p\tau}\di s}{  \int\limits_{\p\Om} e^{-\frac{\tht}{\e^2}} \frac{\p \Psi_\e}{\p\tau}\di s}=& -\frac{\eta_1(\e)}{\mu_0(\e)}
+ \e^{-2} \left(
\frac{\mu_1(\e)\eta_1(\e)}{\mu_0^2(\e)} - \frac{\eta_2(\e)}{\mu_0(\e)}
\right)
\\
&+ \sum\limits_{j=1}^N \e^{2j} K_j^{(pow)}
\left(\frac{\eta_1(\e)}{\mu_0(\e)},\ldots,\frac{\eta_{j+1}(\e)}{\mu_0(\e)}, \frac{\mu_1(\e)}{\mu_0(\e)},\ldots,\frac{\mu_j(\e)}{\mu_0(\e)}\right) + O(\e^{2N+2}),
\end{align*}
where $K_j^{(pow)}$ are some polynomials with fixed coefficients. This identity and estimate (\ref{6.7}) imply asymptotics (\ref{6.15}) for the function $K_\e^{(pow)}$ defined by (\ref{6.12}).

\section*{Acknowledgements}

The authors thanks L.A.~Kalyakin and Yu.A.~Kordyukov for useful discussions and valuable remarks. The reported study  was funded by RFBR according to the research project no.   18-01-00046.

\end{document}